\documentclass[11pt,reqno]{amsart}
\usepackage{amsmath,amsthm,amsfonts,amssymb,bm,graphicx}
\usepackage[alphabetic,initials]{amsrefs}
\usepackage{mathrsfs}
\usepackage{hyperref}
\hypersetup{colorlinks=true, citecolor=blue, linkcolor=blue, urlcolor=blue, pdfstartview=FitH}

\usepackage[top=1.1in,bottom=1.1in,left=1.1in,right=1.1in]{geometry}

\theoremstyle{plain}
\newtheorem{theorem}{Theorem}[section]

\newtheorem{lemma}{Lemma}[section]
\newtheorem{proposition}{Proposition}[section]
\newtheorem{definition}{Definition}[section]

\theoremstyle{definition}
\newtheorem*{Acknowledgements}{Acknowledgement}
\newtheorem{remark}{Remark}[section]

\renewcommand{\geq}{\geqslant}
\renewcommand{\leq}{\leqslant}

\DeclareMathOperator{\GL}{GL}

\newcommand{\Z}{\mathbb{Z}}

\newcommand{\dd}{\mathrm{d}}

\numberwithin{equation}{section}

\makeatletter
\def\Ddots{\mathinner{\mkern1mu\raise\p@
\vbox{\kern7\p@\hbox{.}}\mkern2mu
\raise4\p@\hbox{.}\mkern2mu\raise7\p@\hbox{.}\mkern1mu}}
\makeatother

\makeatletter
\DeclareRobustCommand\widecheck[1]{{\mathpalette\@widecheck{#1}}}
\def\@widecheck#1#2{%
    \setbox\z@\hbox{\m@th$#1#2$}%
    \setbox\tw@\hbox{\m@th$#1%
       \widehat{%
          \vrule\@width\z@\@height\ht\z@
          \vrule\@height\z@\@width\wd\z@}$}%
    \dp\tw@-\ht\z@
    \@tempdima\ht\z@ \advance\@tempdima2\ht\tw@ \divide\@tempdima\thr@@
    \setbox\tw@\hbox{%
       \raise\@tempdima\hbox{\scalebox{1}[-1]{\lower\@tempdima\box
\tw@}}}%
    {\ooalign{\box\tw@ \cr \box\z@}}}
\makeatother

\begin{document}

\title[Subconvexity for $\rm GL_2 \times GL_2$ $L$-functions]
{Subconvexity for $\rm GL_2 \times GL_2$ $L$-functions in the depth aspect}
\author{Tengyou Zhu}
\address{School of Mathematics, Shandong University
                       \\Jinan, Shandong 250100, China}
\email{zhuty@mail.sdu.edu.cn}	

\date{}



\begin{abstract}
Let $f$ and $g$ be holomorphic or Maass cusp forms for $\rm SL_2(\mathbb{Z})$
and let $\chi$ be a primitive Dirichlet character of prime power conductor $q=p^n$. 
For any given $\varepsilon>0$, we establish the following subconvexity bound
\begin{equation*}
L(1/2,f\otimes g \otimes \chi)\ll_{f,g,\varepsilon}q^{9/10+\varepsilon}.
\end{equation*}
The proof employs the DFI circle method with standard manipulations, including the conductor-lowering mechanism, Voronoi summation, and Cauchy--Schwarz inequality. The key input is certain estimates on the resulting character sums, obtained using the $p$-adic version of the van der Corput method.
\end{abstract}

\maketitle

\section{Introduction}
The analytic behavior of $L$-functions on the critical line encodes essential arithmetic information, and the rate of growth along the critical line is of central importance in analytic number theory.
Let $L(s,f)$ be a general $L$-function with an analytic conductor $\mathcal{Q}$.
In many applications of number theory, it is necessary to go beyond the convexity bound to get bounds of the form
\begin{equation}\label{11}
L(1/2, f)\ll \mathcal{Q}^{1/4-\delta+\varepsilon}
\end{equation}
for some absolute constant $\delta>0$.

The Rankin--Selberg convolution stands as a fundamental construction in the theory of 
$L$-functions, with its associated $L$-functions playing a pivotal role in modern number theory. These functions have demonstrated remarkable efficacy in resolving central problems across diverse areas, including establishing bounds for the generalized Ramanujan and Selberg conjectures, addressing the quantum unique ergodicity conjecture, and proving the strong multiplicity one theorem, and so on.
Particularly significant progress has been made in studying subconvexity estimates for 
${\rm GL}_2\times {\rm GL}_2$ Rankin--Selberg  $L$-functions (see~\cite{ASS}, \cite{HM}, \cite{HT}, \cite{KMV}).
In this paper, we investigate subconvexity bounds for ${\rm GL}_2\times {\rm GL}_2$ Rankin--Selberg  $L$-functions on the critical line in the depth aspect.

Let $f$ and $g$ be holomorphic or Maass cusp forms for $\rm SL_2(\mathbb{Z})$
 with normalized Fourier coefficients
$\lambda_f(m)$ and $\lambda_g(m)$ (i.e., $\lambda_f(1)=1$ and $\lambda_g(1)=1$), respectively.
Let $\chi$ be a primitive Dirichlet character of prime power conductor
$q=p^n$, where $p$ is prime and $n$ is an integer.
The $L$-function associated with $f$ and $g\otimes\chi$ is given by
$$
L(s,f\otimes g \otimes \chi)=L\big(2s,\chi^2\big)\sum_{m\geq1}
\frac{\lambda_f(m)\lambda_g(m)\chi(m)}{m^s}
$$
for $\text{Re} (s) > 1 $, which can be analytically extended to $\mathbb{C}$ and satisfies
a functional equation relating $s$ and $1-s$. The approximate functional equation
and the Phragm\'en-Lindel\"{o}f principle imply that
$$L(1/2,f\otimes g \otimes \chi)\ll_{f,g,\varepsilon}q^{1+\varepsilon}$$
for any $\varepsilon>0$, which is the convexity bound in the depth aspect.
The aim of the present paper is to prove the following subconvexity bound
using the delta symbol approach.
\begin{theorem}\label{main-theorem1}
Let $p$ be an odd prime.
Let $f$ and $g$ be holomorphic or Maass cusp forms for $\rm SL_2(\mathbb{Z})$
and let $\chi$ be a primitive Dirichlet character of prime power conductor $q=p^n$. 
For any given $\varepsilon>0$, we have
$$L(1/2,f\otimes g \otimes \chi)\ll_{f, g, \varepsilon}q^{9/10+\varepsilon}.$$
\end{theorem}
This improves the previous result of Sun~\cite{S}, which states 
$$
L(1/2,f\otimes g \otimes \chi)\ll_{f,g,\varepsilon}
p^{3/4}q^{15/16+\varepsilon}.
$$
In the $t$-aspect, Huang, Sun and Zhang~\cite[Corollary 1.5]{HSZ} have proved that
$$L(1/2+it,f\otimes g)\ll_{f,g,\varepsilon} (1+|t|)^{9/10+\varepsilon}.$$
Indeed, in principle the family $L(1/2,f\otimes g \otimes \chi)$ where $\chi$ runs over
characters modulo $p^n$ with $n\rightarrow\infty$ should behave like $L(1/2+it,f\otimes g)$
with $t\rightarrow\infty$. Since, both problems are arithmetic, we end up seeking
square root cancellations. Huang, Sun and Zhang~\cite{HSZ} had to appeal to the 
classical second derivative test for exponential integrals to show square root cancellations. 
In our case, we achieve the required cancellations 
by using the $p$-adic second derivative test 
(\cite[Theorem~5]{BM}) for exponential sums.
\begin{remark}
Since the holomorphic case is usually easier,
to simplify the argument, we prove Theorem~\ref{main-theorem1} only for the case of Maass forms.
In fact, the Voronoi formula for Maass forms is
slightly more complicated than that of the holomorphic case.
\end{remark}
\begin{remark}
Set $\mathcal{Q}=\big(q(1+|t|)\big)^4$, the hybrid subconvexity bounds  
$$L(1/2+it, f\otimes g \otimes \chi)\ll_{f, g, \varepsilon} \mathcal{Q}^{1/4-1/40+\varepsilon}$$
can be proved by combining our approach and that of Huang, Sun and Zhang~\cite{HSZ}, 
but this requires a rigorous proof. The exponent $\delta=1/40$ in~\eqref{11} appears in other contexts as well and it seems to be the limit of the delta symbol approach.
\end{remark}

\subsection{A brief history}
We will briefly recall the history of the problem but focus only on the in-depth aspects.
Let $\chi$ be a primitive Dirichlet character of prime power conductor $q=p^n$.
In the ${\rm GL}_1$ case, by introducing a general new theory of estimation of
short exponential sums involving $p$-adic analytic phases,
Mili\'{c}evi\'{c} \cite{MD} proved a sub-Weyl subconvexity bound for the central values $L(1/2,\chi)$
of Dirichlet $L$-functions of the form
$$
L(1/2,\chi)\ll p^v q^{\varpi}(\log q)^{1/2}
$$
with a fixed $v$ and $\varpi>\varpi_0\approx 0.1645$.
By developing a general result on $p$-adic approximation by rationals 
(a $p$-adic counterpart to Farey dissection) 
and a $p$-adic version of van der Corput's method for exponential sums,
Blomer and Mili\'{c}evi\'{c} \cite{BM}
also considered the ${\rm GL}_2$ case. For $g$ being a holomorphic or 
Maass newform for ${\rm SL}_2(\mathbb{Z})$, they showed that 
$$
L\left(1/2+it, g\otimes\chi\right)\ll_{g,\varepsilon}(1+|t|)^{5/2}p^{7/6}q^{1/3+\varepsilon},
$$
where the implied constant on $p$ and $t$ is explicit and polynomial. 
Using the conductor-lowering trick introduced by Munshi in~\cite{Mun3}, 
Munshi and Singh~\cite{MS} proved the same result using the approach in \cite{Mun3}.
Sun and Zhao~\cite{SZ} extended Munshi and Singh's results
in \cite{MS} to the ${\rm GL}_3$ case and proved that for
$\pi$ being a Hecke-Maass cusp form for ${\rm SL}_3(\mathbb{Z})$,
$$
L\left(1/2,\pi\otimes \chi\right)\ll_{\pi,\varepsilon} p^{3/4}q^{3/4-3/40+\varepsilon}
$$
for any $\varepsilon>0$. Sun~\cite{S} also establishes 
a subconvexity bound for the $\rm GL_2 \times GL_2$ case. 
The ideas of treating character sums in the paper~\cite{SZ} and~\cite{S} can be
applied to prove a ${\rm GL}_3\times {\rm GL}_2$ subconvexity bound of the form 
$$
L\left(1/2,\pi \otimes f \otimes \chi\right)\ll_{\pi,f,\varepsilon} q^{3/2-3/20+\varepsilon},
$$
where $\pi$ is a Hecke-Maass cusp form for ${\rm SL}_3(\mathbb{Z})$, $f$ is a holomorphic or Maass cusp form for ${\rm SL}_2(\mathbb{Z})$. This result has been proved by Kumar, Mallesham and Singh~\cite{KMS}.

\subsection{Principle of the proof}
To illustrate the main idea of our approach, we provide 
a quick sketch of the proof for Theorem~\ref{main-theorem1}.
By the approximate functional equation, our starting point is to estimate
$$
S(N)=\sum_{m\sim N}\lambda_f(m)\lambda_g(m)\chi(m),
$$
with $N\sim q^2$. Applying the conductor-lowering mechanism
introduced by Munshi \cite{Mun1}, we have
$$
S(N)=\mathop{\sum\sum}_{\substack{\ell, m \sim N \\ \ell\equiv m \bmod p^r}}
\lambda_g(\ell)\lambda_f(m)\chi(m)\delta\bigg(\frac{m-\ell}{p^r}\bigg),
$$
where $r \geq 2$ is an integer to be chosen later. After 
applying Duke, Friedlander and Iwaniec's delta method with $c \leq C$ (see~\eqref{DFI's}) 
and removing the congruence $\ell\equiv m \bmod p^r$ by exponential sums we get
$$
S(N)\approx \frac{1}{Cp^r}\sum_{\substack{c\sim C \\ (c, p)=1}}\frac{1}{c}
\;\sideset{}{^*}\sum\limits_{a\bmod cp^r}\mathop{\sum\sum}_{\ell, m \sim N}
\lambda_g(\ell)\lambda_f(m)\chi(m)e\bigg(\frac{a(m-\ell)}{cp^r}\bigg),
$$
where here and throughout, the $*$ on the sum $\sum_{a \bmod c}$ over $a$ means that the sum is restricted to $(a,c)=1$, and we take $C=\sqrt{N/p^r}$. By using the Ramanujan conjecture on average, trivially we have $S(N)\ll N^2$.

We apply the voronoi summation formulas to both $m$ and $\ell$ sums.
Recall that $\chi$ is of modulus $q=p^n$.
Then the conductor of the $m$-sum has the size $Cp^n$. Applying $\rm GL_2$ Voronoi summation
to the $m$-sum we get that the dual sum is of size $C^2p^{2n}/N\asymp p^{2n-r}$.
The conductor for the $\ell$-sum has the size $cp^r$ and the dual sum after
$\rm GL_2$ Voronoi summation
is essentially supported on summation of size $C^2p^{2r}/N\asymp p^{r}$. 
Due to the presence of a Ramanujan sum in the character sum and the assumption of square-root cancellation for the remaining component, we achieve a saving of
$$
\frac{N}{Cp^n}\times \frac{N}{Cp^r}\times C p^{r/2}\sim N.
$$
Hence we are at the threshold and need to save a little more.
Generally we arrive at an expression of the form
$$
\sum_{c\sim C}\sum_{m\sim p^{2n-r}}\lambda_f(m)
\sum_{\substack{\ell \sim p^r \\\ell \equiv p^{r}\overline{p^{2n-r}} m \bmod c}}
\lambda_g(\ell)G(m,\ell,c),
$$
where
\begin{align*} 
G(m,\ell,c)=\frac{1}{q}\sum_{u \bmod p^n} \chi (u) 
S\big(u\overline{c}, m\overline{c}; p^n\big)
S\big(u \overline{c}, \ell\overline{c}; p^r\big).
\end{align*}
Next we apply the Cauchy--Schwarz inequality to get rid of the Fourier coefficients.
Then we need to deal with
$$
\sum_{c\sim C}\sum_{m\sim p^{2n-r}}\bigg|
\sum_{\substack{\ell \sim p^r \\\ell \equiv p^{r}\overline{p^{2n-r}} m \bmod c}}
\lambda_g(\ell)G(m,\ell,c)\bigg|^2.
$$
After opening the square and applying Poisson summation to the sum over 
$m$, we obtain two dominant terms: Diagonal term (from $\ell_1=\ell_2$) with a
saving factor of $p^r/C$. Off-diagonal term (from $\ell_1 \not =\ell_2$) with a
saving factor of
$$
\frac{p^{n-2r}}{C} \times \frac{p^{n+r}}{Cp^r}= \frac{p^{n-2r}}{Cp^{r/2}}.
$$
Therefore, we are able to save $p^r/C\sim p^{3r/2-n}$ from the diagonal term
and $p^{2n-r}/(Cp^{r/2})\sim p^{n-r}$
from the off-diagonal term. By setting the two savings to be of the same order,
the optimal choice for $r$ is given by $r=4n/5$. In total, we have saved
$N\times q^{1/10}$. It follows that
$$
L(1/2,f\otimes g \otimes \chi)\ll
N^{-1/2}S(N)\ll N^{1/2}q^{-1/10}\sim q^{1-1/10}.
$$
In the following sections we shall provide the proof of the theorem in detail.

\begin{Acknowledgements}
I would like to thank Qingfeng Sun, who suggested this problem to me, for her encouragement. 
I would also like to Djordje Mili\'cevi\'c for his very helpful discussions and suggestions.
\end{Acknowledgements}

\section{Preliminaries and Notation}
\label{prelim}

In this section we set up some notation and compile for the future reference a 
number of useful results, some of which are well-known.

\subsection{Notation}
We denote by $\mathbb{Z}_p$ the ring of $p$-adic integers in the field $\mathbb{Q}_p$ of $p$-adic numbers, and
by $\mathbb{Z}_p^{\times} = \mathbb{Z}_p \setminus p\mathbb{Z}_p$, the group of its units.
Then $\mathbb{Q}_p^{\times} = \mathbb{Q}_p \setminus \{0\}
= \bigsqcup_{k\in \mathbb{Z}}p^k\mathbb{Z}_p^{\times}$; for $x \in p^k\mathbb{Z}_p^{\times}$,
we write $\text{ord}_px=k$.

We write $e(x) = e^{2\pi i x}$ for $x\in \mathbb{R}$, and we write $\theta : \mathbb{Q}_p \rightarrow 
\mathbb{C}^\times$ for the standard additive character on $\mathbb{Q}_p$ trivial on $\mathbb{Z}_p$. 
Specifically, if $x=\sum_{j\geq j_0}a_jp^j\in \mathbb{Q}_p$, then
$\theta(x) = \exp\big(2\pi i \sum_{j < j_0}a_jp^j\big)$. 
In particular, on $\mathbb{Z}[1/p] \subseteq \mathbb{Q}_p \cap \mathbb{R}$ we have $\theta(\cdot) = e(\cdot)$, and
this relation is crucial for moving arithmetic oscillation between $p$-adic and archimedean
places.

If $k \geq 0$ and a domain $T\subseteq\mathbb{Z}_p$ are 
such that $(T+p^k\mathbb{Z}_p)\subseteq T$, and if $\psi$ is
a $p^k\mathbb{Z}_p$-periodic function, then by
\begin{equation*}
\sum_{u \bmod p^k, u\in T}\psi(u),
\end{equation*}
where the (finite) sum of $\psi(u)$ over arbitrary set of representatives
of classes in $T/p^k\mathbb{Z}_p$.

For $x\in \mathbb{Z}$ coprime to the modulus (or sometimes for $x\in \mathbb{Z}_p^{\times}$), 
the notation $\overline{x}$ will always
denote the multiplicative inverse to a modulus which will be obvious from the context. 
In particular, if no obvious modulus is specified, then $\bar{x}=x^{-1}$
will denote the inverse of $x$ in $\mathbb{Z}_p^{\times}$,
which agrees with the multiplicative inverse of $x$ to any prime power $p^n$. 
While both $\overline{x}$ and $x^{-1}$ have the same meaning for
for aesthetic reasons we usually give preference to the
former in notations that explicitly depend only on the congruence class of $x$ to an obvious
modulus (such as in factors, Kloosterman sums, and so on) and to the latter in expressions
of a more direct $p$-adic nature such as phases in the context of $p$-adic local analysis (even
when the expression happens to be locally constant).

We introduce the following terminology, including an auxiliary class of functions.
\begin{definition}\label{MM}
Let $\kappa \in \mathbb{Z}$. We denote by $\mathbf{M}_{p^{\kappa}}$ an 
arbitrary element of $p^\kappa\mathbb{Z}_p$, which may
be different from line to line.  
\end{definition}

\subsection{Postnikov Formula}

We recall the structure of multiplicative characters modulo $q=p^n$.
In this section, as everywhere else in the paper, $p$ denotes an odd prime;
all statements hold with minor but necessary modifications in the case $p=2$.
\begin{definition}
The $p$-adic logarithm, $\log_p : 1+p\mathbb{Z}_p \to p\mathbb{Z}_p$ is the analytic function given as
$$
\log_p (1 + x) := \sum_{j \geqslant 1} (-1)^{j-1} \frac{x^j}{j}.
$$
\end{definition}
Access to the above is critical due to the following lemma, 
with roots in Postnikov and which we quote from~\cite[Lemma 13]{MD}.
\begin{lemma}\label{Postnikov}
Let $\chi$ be a primitive character modulo $p^n$. Then there exists a $p$-adic unit $\alpha$ such that, for every $m \equiv 1 \bmod p$,
\begin{equation}\label{Postnikov formula}
\chi(m) = \theta \left( \frac{\alpha \log_p m}{p^n} \right).
\end{equation}
\end{lemma}
In fact, we do not need the full strength of~\eqref{Postnikov formula}. Form~\cite[(2.3)]{BM}, we will use the following
corollary, valid for every $\kappa \geq 1$ and every $u \in \Z_p^{\times}$, $t \in \Z_p$: 
\begin{equation}\label{chi expansion}
\chi(u+p^\kappa t) = \chi(u)\chi(1+p^{\kappa}u^{-1}t)
=\chi(u)\theta\left(\frac{\alpha}{p^n}\Big(\frac{1}{u}p^\kappa t
-\frac{1}{2u^2}p^{2\kappa} t^2\Big) + \mathbf{M}_{p^{3\kappa-n-\iota}} \right),
\end{equation}
in the sense of Definition~\ref{MM}. Here we denote $\iota=1$ if $p=3$ and $\iota=0$ otherwise.

\subsection{$p$-adic square roots}

It will be necessary to handle solutions to quadratic equations over $\Z_p$, which requires the use of $p$-adic square roots. While these square roots naturally arise in $p$-adic towers as
in~\cite[Section 2.4]{BM}, we keep our exposition elementary and only discuss square roots to a prime
power modulus $p^n$. For $p$ an odd prime and $x \in \Z_p^{\times 2}$, the congruence $u^2 \equiv x \bmod{p^n}$ has exactly two solutions modulo every $p^n$, which reside within two $p$-adic towers and limit to the solutions of $u^2=x$ as $n \to \infty$. We denote these solutions $\pm x_{1/2}$. For $(\, \cdot \,)_{1/2} :\Z_p^{\times 2}\to\Z_p^{\times}$ to be well-defined, a choice of square root for each $y\in (\Z / p \Z)^{\times 2}$ must be made. This set of choices propagates to $\Z_p^{\times 2}$ and represents one of the $2^{(p-1)/2}$ branches of the $p$-adic square root. Finally, $x/x_{1/2}$ and $(1/x_{1/2})^2=1/x$ for every $x \in \Z_p^{\times 2}$, simply because $x_{1/2}^2=x$.
 
Next we note that, for every $\kappa \geqslant 1$ and every $u \in \Z_p^{\times 2}$, $t \in \Z_p^{\times}$,
\begin{equation*}
\left(u_{1/2}+\frac{1}{2u_{1/2}}p^\kappa t - \frac{1}{8u_{1/2}^3}
p^{2\kappa} t^2\right)^2 \in (u+p^\kappa t)+ p^{3\kappa}\Z_p,
\end{equation*}
so that
\begin{equation*}
(u+p^\kappa t)_{1/2} \equiv u_{1/2}+\frac{1}{2u_{1/2}}p^\kappa t 
- \frac{1}{8u_{1/2}^3} p^{2\kappa} t^2  \;\bmod\; p^{3\kappa}.
\end{equation*}
With the terminology of Definition~\ref{MM}, we claim that actually 
\begin{equation}\label{square power series expansion}
(u+p^\kappa t)_{1/2} \equiv u_{1/2}+\frac{1}{2u_{1/2}}p^\kappa t 
- \frac{1}{8u_{1/2}^3} p^{2\kappa} t^2 +\mathbf{M}_{p^{3\kappa}}.
\end{equation}
For future reference, we note that, for all $u_1, u_2\in\Z_p^{\times 2}$,
\begin{equation}\label{ord-sqrt}
\text{ord}_p\big((u_1)_{1/2}-(u_2)_{1/2}\big)=\text{ord}_p(u_1-u_2).
\end{equation}

\subsection{Method of stationary phase}

In this section, we collect facts about the so-called $p$-adic method of stationary phase, a powerful tool in the study of complete exponential sums modulo prime powers analogous to the classical method of stationary phase for oscillatory exponential integrals. For more details, we refer to \cite[Lemmata~12.2 and~12.3]{IK} for a formulation with phases that are rational functions, or to \cite[Lemma~7]{BM} for a general statement.

Let $p > 2$ be a prime. For $s \in \{0, 1\}$ and $(A, p) = 1$, we define
$$\epsilon(A,p^s)=
\begin{cases}
1,                       &  {s =  0,}\\
\big(\frac{A}{p}\big),   &  {s =  1, \; p\equiv 1 \bmod 4,}\\
\big(\frac{A}{p}\big)i,  &  {s =  1, \; p\equiv 3 \bmod 4.}
\end{cases}$$
\begin{lemma}[$p$-adic stationary phase]\label{statphase-lemma}
Let $n \geq \kappa$, and let $T\subseteq\mathbb{Z}_p$ be such 
that $(T+p^\kappa\mathbb{Z}_p)\subseteq T$.
\begin{enumerate}
\item\label{MSP-claim1} Suppose that functions $\psi:T/p^n\mathbb{Z}_p \to \mathbb{C}^\times,\, 
\phi_1:T\to\mathbb{Q}_p$ satisfy
\begin{equation}
\label{diffble-eq1}
\psi(u+p^{\kappa}t) = \psi(u)\theta(\phi_1(u)\cdot p^{\kappa}t)
\end{equation}
for all $u\in T$ and $t\in\mathbb{Z}_p$. Then, we have
\begin{equation*}
\sum_{u \bmod p^n, u\in T}\psi(u)
=p^{n-\kappa}\sum_{\substack{u \bmod p^\kappa, u\in T \\ \phi_1(u)\in p^{-\kappa} \mathbb{Z}_p}}\psi(u).
\end{equation*}
\item\label{MSP-claim2}
Suppose that functions $\psi:T/p^n\mathbb{Z}_p \to \mathbb{C}^\times,\, 
\phi_1, \phi_2:T\to\mathbb{Q}_p$ satisfy
\begin{equation}\label{diffble-eq2}
\psi(u+p^{\kappa}t) = \psi(u)\theta\left(\phi_1(u)\cdot p^{\kappa}t
+\frac{1}{2}\phi_2(u)\cdot p^{2\kappa}t^2\right)
\end{equation}
for all $u\in T$ and $t\in\mathbb{Z}_p$. 
Assume that ${\rm ord}_p \phi_1(u) \geq -n$ and ${\rm ord}_p \phi_2(u) = \mu$ 
for every $u\in T$, and that $-2n\leq \mu \leq -2\kappa$.
Then, writing $\mu=-2r-\rho$ with $r\in \mathbb{Z}$ and $\rho\in\{0,1\}$,
\begin{equation*}
\sum_{u \bmod p^n, u\in T}\psi(u)=p^{n+\mu/2}
\sum_{\substack{u \bmod p^r, u\in T \\ \phi_1(u)\in p^{-r-\rho} \mathbb{Z}_p}}
\psi(x)\epsilon\big(\bar{2}(\phi_2(u))_0,p^{\rho}\big)
\theta\bigg(-\frac{{\phi_1(u)^2}}{2\phi_2(u)}\bigg).
\end{equation*}
\end{enumerate}
\end{lemma}

\begin{proof}
The first claim follows immediately by orthogonality from
\begin{align*}
\sum_{u \bmod p^n, u\in T}\psi(u)\;&=\sum_{u \bmod p^\kappa, u\in T}\sum_{t \bmod p^{n-\kappa}}\psi(u+p^\kappa t)\\
\;&=\sum_{u \bmod p^\kappa, u\in T}\psi(u)\sum_{t \bmod p^{n-\kappa}}\theta(\phi_1(u)\cdot p^{\kappa}t). 
\end{align*}
The second claim is immediate if $n=2\kappa$. The odd case is similar, using the classical evaluation of the quadratic Gauss sum; see~\cite[Lemma~7]{BM}.
\end{proof}

\begin{lemma}[Kloosterman sum evaluation]\label{kloost-eval}
Let $p$ be an odd prime, let $a, b\in \Z_p^{\times}$, let $n \geqslant 2$ 
and let  
\begin{equation*}
S\big(a,b;p^n\big)=\;\sideset{}{^*}\sum_{u\bmod{p^n}}
 \theta\left(\frac{au+b/u}{p^n}\right)
\end{equation*} 
be the Kloosterman sum. Then,  $S\big(a,b;p^n\big)=0$ 
if $a b \not\in \Z_p^{\times 2}$. 
Otherwise, if $a b \in \Z_p^{\times 2}$, denoting $\rho=0$ or
$1$ according to whether $n$ is even or odd,
\begin{equation*}
S\big(a, b;p^n\big) = p^{n/2}\sum_{\pm}\epsilon(\pm(a b)_{1/2}, 
p^{\rho})\theta\left(\pm \frac{2(a b)_{1/2}}{p^n}\right).
\end{equation*}
\end{lemma}
\begin{proof}
See~\cite[Lemma~8]{BM}.
\end{proof}

\begin{lemma}[Second derivative test]\label{Second derivative test}
Let $\omega \in \mathbb{R}$, let $\kappa_0, \upsilon, \lambda \in \mathbb{N}_0$
and $T \in \mathbb{Z}_p$ be such that $(T+p^{\kappa_0}\mathbb{Z}_p) \in T$, and let
$\psi:T \to \mathbb{C}^\times$, 
$\phi_1:T\to p^{-\upsilon}\mathbb{Z}_p$, 
$\phi_2:T\to p^{-\lambda}\mathbb{Z}_p^{\times}$, and $$\Omega:\{\kappa\in \mathbb{N}_0; \kappa\geq \kappa_0\}
\to \mathbb{N}_0$$ be functions such that $|\psi(u)|\leq \psi_0$,
\begin{equation}
\psi(u+p^{\kappa}t) = \psi(u)\theta\left(\phi_1(u)\cdot p^{\kappa}t
+\frac{1}{2}\phi_2(u)\cdot p^{2\kappa}t^2+\mathbf{M}_{p^{\Omega(\kappa)}}\right)
\end{equation}
for all $u\in T$, $t\in\mathbb{Z}_p$, $\kappa\geqslant\kappa_0$, and 
$$\Omega(\kappa) \geq \min (2\kappa-\lambda+1,0) \quad \text{for every} \ \kappa\geq \kappa_0.$$ 
Let $\kappa_1=\max(\lambda/2, \kappa_0)$, then we have
\begin{equation*}
\sum_{u \bmod p^n, u\in T}\psi(u)e(\omega u)\ll
\psi_0 (p^n+p^{\upsilon}+p^{\lambda})p^{\min(\kappa_1-\lambda+\kappa_0,0)}.
\end{equation*}
\end{lemma}

\begin{proof}
See~\cite[Theorem~5]{BM}.
\end{proof}

\subsection{Maass cusp forms for $\mathrm{GL}_2$}

We briefly review some basic facts of automorphic $L$-functions on $\rm \GL_2$.
Let $f$ be a Hecke-Maass cusp form for $\rm SL_2(\mathbb{Z})$
with Laplace eigenvalue $1/4+\mu_f^2$. Then $f$ has a Fourier expansion
$$
f(z)=\sqrt{y}\sum_{m\neq 0}\lambda_f(m)K_{i\mu_f}(2\pi |m|y)e(mx),
$$
where $K_{i\mu_f}$ is the modified Bessel function of the third kind.
We need the following average bound
\begin{align}\label{GL2: Rankin Selberg}
\sum_{m\leq N}|\lambda_f(m)|^2= \mu_{f} N+O\big(N^{3/5}\big).
\end{align}

A ubiquitous tool in the analysis of $L$-functions inside the critical strip
is the approximate functional equation (see \cite[\S 5.2]{IK}). 
This equation has various manifestations depending on context and purpose. 
For our purposes, the following lemma is convenient, 
which follows by applying a dyadic partition of unity to \cite[Theorem~5.3]{IK}.
\begin{lemma}\label{AFE}
Let $\chi$ be a primitive Dirichlet character modulo $q$. Then we have
\begin{equation*}
L(1/2,f\otimes g \otimes \chi)\ll_{f,g,\varepsilon} q^{\varepsilon}
\sup_{N\leq q^{2+\varepsilon}}\frac{|S(N)|}{\sqrt{N}}+q^{-A},
\end{equation*}
where
$$
S(N)=\sum_{m\geq1}\lambda_f(m)\lambda_g(m)\chi(m)
V\left(\frac{m}{N}\right)
$$
for some smooth function $V$ supported in $[1,2]$ and satisfying $V^{(j)}(x)\ll_j 1$.
\end{lemma}

Now we turn to the Voronoi summation formula for $\rm SL_2(\mathbb{Z})$ 
(see \cite[(1.12) and (1.15)]{MS} and \cite[Theorem A.4]{KMV}).  
Define
\begin{align}\label{Gamma1}
\gamma_f^{\pm} (s)=\frac{1}{2\pi^{1+2s}}\left\{\prod_{\pm}
\frac{\Gamma\left(\frac{1+s\pm\mu_f}{2}\right)}
{\Gamma\left(\frac{-s\pm\mu_f}{2}\right)}\mp
 \prod_{\pm}\frac{\Gamma\left(\frac{2+s\pm\mu_f}{2}\right)}
{\Gamma\left(\frac{1-s\pm\mu_f}{2}\right)}\right\}.
\end{align}
Then we have the following Voronoi summation formula.
\begin{lemma}\label{voronoiGL2-Maass}
For $h(x)\in C_c^\infty(0,\infty)$, we denote by $\widetilde{h}(s)$ the Mellin transform of $h(x)$.
Let $a, \overline{a}, c \in \mathbb{Z}$ with $c\neq 0$, $(a,c)=1$ 
and $a \overline{a} \equiv 1\bmod c$. For $N>0$, we have
\begin{align*}
\sum_{m \geq 1} \lambda_f(m)e\left(\frac{am}{c}\right)h\left(\frac{m}{N}\right)
=\frac{N}{c} \sum_{\pm}\sum_{m \geq 1} \lambda_f(m)
e\left(\mp\frac{\overline{a}m}{c}\right)\Psi_h^{\pm}\left(\frac{mN}{c^2}\right),
\end{align*}
where for $\sigma>-1$,
\begin{align}\label{intgeral transform-1}
\Psi_h^{\pm}(x)
&=\frac{1}{2\pi i}\int_{(\sigma)}x^{-s}
\gamma_f^{\pm}(s)\widetilde{h}(-s)\mathrm{d}s \\
&\label{intgeral transform-2}
= \int_0^\infty h(y) \mathcal{J}_f^{\pm}(4\pi\sqrt{xy})\dd y,
\end{align}
with
\begin{align*}
\mathcal{J}_f^+(x)=\frac{-\pi}{\sin(\pi i \mu_f)} \left(J_{2i \mu_f}(x)
- J_{-2i \mu_f}(x)\right) ,
\end{align*}
and
\begin{align*}
\mathcal{J}_f^-(x)=4\varepsilon_f \cosh(\pi \mu_f) K_{2i \mu_f}(x).
\end{align*}
\end{lemma}
The function $\Psi_h^{\pm}(x)$ has the following asymptotic expansion when  $x\gg 1$.
\begin{lemma}\label{voronoiGL2-Maass-asymptotic}
For $x\gg 1$, we have
\begin{align*}
\Psi_h^{+}(x)=x^{-1/4} \int_0^\infty h(y)y^{-1/4}
\sum_{j=0}^{J}
\frac{c_{j} e(2 \sqrt{xy})+d_{j} e(-2 \sqrt{xy})}
{(xy)^{j/2}}\mathrm{d}y
+O_{t_f,J}\left(x^{-J/2-3/4}\right),
\end{align*}
for some any fixed integer $J \geq 1$ and
\begin{align*}
\Psi_h^-(x)\ll_{t_f,A}x^{-A}.
\end{align*}
\end{lemma}
\begin{proof}
See \cite[Lemma 3.4]{LS}.
\end{proof}

\begin{remark}\label{JK remark}
Notice that the above Lemma are only valid for $x\gg 1$. So we also need
the facts which state that, for $x>0$, $k\geq 0$ and $\text{Re} (\nu) = 0$. One has see \cite[Lemma C.2]{KMV}
\begin{align*}
y^{k}J_\nu^{(k)}(x)\ll_{k,\nu}\frac{1}{(1+x)^{1/2}}, \qquad
y^{k}K_\nu^{(k)}(x)\ll_{k,\nu}\frac{e^{-x}(1+|\log x|)}{(1+x)^{1/2}}.
\end{align*}
\end{remark}

\section{The set-up}

We recall the set-up and our general assumptions that will be in force for the rest of the paper.
Let $\alpha \in \mathbb{Z}_p^{\times}$ be such that~\eqref{Postnikov formula} 
holds for the character $\chi$ modulo $q=p^n$ in Theorem~\ref{main-theorem1}. 

Now we start to prove Theorem~\ref{main-theorem1}. Recall that, Lemma~\ref{AFE}, 
we are considering $S(N)$ with $N\ll q^{2+\varepsilon}$. We shall establish the following bound.
\begin{proposition}\label{main prop}
Assume $n/2 \leq r < n$. Then we have 
\begin{equation*}
S(N)\ll N^{3/4+\varepsilon}\big(p^{n-3r/4}+p^{r/2}\big).
\end{equation*}  
\end{proposition}

We take $r=\lfloor 4n/5\rfloor$, where $\lfloor x \rfloor$ denotes the largest integer which 
does not exceed $x$. By Proposition~\ref{main prop}, we have
$$
S(N)\ll N^{1/2+\varepsilon}q^{9/10}
$$
from which Theorem~\ref{main-theorem1} follows. 
The rest of the paper is devoted to prove Proposition~\ref{main prop}.

\subsection{Applying the delta method}

Define $\delta: \mathbb{Z}\rightarrow \{0,1\}$ with $\delta(0)=1$ and $\delta(m)=0$ for $m \neq 0$.
To separate oscillations from $\lambda_g(m)$ and $\lambda_f(m)\chi(m)$,
we will use a version of the circle method by Duke, Friedlander and Iwaniec (see \cite[Chapter 20]{IK}).
For any $m\in \mathbb{Z}$ and $C\in \mathbb{R}^+$, we have
\begin{align}\label{DFI's}
\delta(m=0)=\frac{1}{C}\sum_{1\leq c\leq C} \;\frac{1}{c}\;\sideset{}{^*}\sum_{a\bmod{c}}
e\left(\frac{ma}{c}\right)\int_\mathbb{R}g(c,\zeta) e\left(\frac{m\zeta}{cC}\right)\mathrm{d}\zeta,
\end{align}
where the $*$ on the sum indicates that the sum over $a$ is restricted to $(a,c)=1$.
The function $g$ has the following properties
(see (20.158) and (20.159) of~\cite{IK})
\begin{align}\label{g-h}
g(c,\zeta)\ll |\zeta|^{-A},\quad g(c,\zeta) =1+
O\left(\frac{C}{c}\left(\frac{c}{C}+|\zeta|\right)^A\right)
\end{align}
for any $A>1$ and
\begin{align}\label{g rapid decay}
\frac{\partial^j}{\partial \zeta^j}g(c,\zeta)\ll
|\zeta|^{-j}\min\left(|\zeta|^{-1},\frac{C}{c}\right)\log C, \quad j\geq 1.
\end{align}
In particular the first property in \eqref{g-h} implies that
the effective range of the integration in
\eqref{DFI's} is $[-C^\varepsilon, C^\varepsilon]$.

In order to use the delta symbol method in~\eqref{DFI's}, we proceed as \cite[(3.7)]{LMS}
(see \cite[Lemma 1]{S} and \cite[Lemma 3.2 ]{WZ} for more details) 
by writing $\delta(m=0)$ in a more analytic form:
\begin{align}\label{circle method}
\delta(m=0)=&\sum_{k=1}^{[\log C/\log p]}
\frac{1}{C}\sum_{\substack{1\leq c\leq C/p^k \\ (c,p)=1}}\frac{1}{cp^{r+k}}
\sideset{}{^*}\sum_{a\bmod cp^{r+k}}
e\left(\frac{am}{cp^{r+k}}\right)
\int_{\mathbb{R}} g(p^k c,\zeta)e\left(\frac{m\zeta}{cCp^{r+k}}\right)
 \mathrm{d}\zeta\notag\\
&+\sum_{\lambda=0}^{r}\frac{1}{C}\sum_{\substack{1\leq c\leq C \\ (c,p)=1}}
\frac{1}{cp^r}\;\sideset{}{^*}\sum_{a\bmod cp^{r-\lambda}}
e\left(\frac{am}{cp^{r-\lambda}}\right)\int_{\mathbb{R}} g(c,\zeta)e\left(\frac{m\zeta}{cCp^r}\right)\mathrm{d}\zeta
,
\end{align}
where $r \in \mathbb{Z}$.
Instead of using the entire modulus $q$ for the conductor-lowering mechanism, 
we only use a part $p^r$, where $r < n$ is optimally chosen later. 
This introduces more terms on the diagonal while having less impact on the off-diagonal.

Now we write
$$
S(N)=\sum_{\ell \geq 1}\lambda_g(\ell)W\left(\frac{\ell}{N}\right)
\sum_{\substack{m\geq 1 \\ p^r|m-\ell}}
\lambda_f(m)\chi(m)V\left(\frac{m}{N}\right)
\delta\left(\frac{m-\ell}{p^r}\right),
$$
where $W$ is a smooth function supported in $[1/2,5/2]$, $W(x)=1$
for $x\in [1,2]$ and $W^{(j)}(x)\ll_j 1$.
Applying~\eqref{circle method} with $ C=\sqrt{N/p^r}, $ we have
\begin{align*}
S(N)=\, & \sum_{m \geq 1}\sum_{\ell \geq 1}\lambda_g(\ell)W\left(\frac{\ell}{N}\right)
\lambda_f(m)\chi (m)V\left(\frac{m}{N}\right) \\
&\notag\times\bigg\{\sum_{\lambda=0}^{r}\frac{1}{C}\sum_{\substack{1\leq c\leq C \\ (c,p)=1}}
\frac{1}{cp^r}\;\sideset{}{^*}\sum_{a\bmod cp^{r-\lambda}}
e\left(\frac{a(m-\ell)}{cp^{r-\lambda}}\right)\int_{\mathbb{R}} g(c,\zeta)e\left(\frac{(m-\ell)\zeta}{cCp^r}\right)\mathrm{d}\zeta\\
&\notag+\sum_{k=1}^{[\log C/\log p]}
\frac{1}{C}\sum_{\substack{1\leq c\leq C \\ (c,p)=1}}\frac{1}{cp^{r+k}}
\sideset{}{^*}\sum_{a\bmod cp^{r+k}}e\left(\frac{a(m-\ell)}{cp^{r+k}}\right)
\int_{\mathbb{R}} g(p^k c,\zeta)e\left(\frac{(m-\ell)\zeta}{cCp^{r+k}}\right)
 \mathrm{d}\zeta\bigg\}.
\end{align*}
In the above sum, we only consider the first term in the braces with $\lambda = 0$, that is,
\begin{align*}
\widetilde{S}(N):= & \, \sum_{m \geq 1}\sum_{\ell \geq 1}
\lambda_g(\ell)W\Big(\frac{\ell}{N}\Big)
\lambda_f(m) \chi (m)  V\Big(\frac{m}{N}\Big)\\
&\times\frac{1}{C}\sum_{\substack{1\leq c\leq C \\ (c,p)=1}}
\frac{1}{cp^r}\;\sideset{}{^*}\sum_{a\bmod{cp^r}}
e\left(\frac{(m-\ell)a}{cp^r}\right)
\int_\mathbb{R}g(c,\zeta) e\left(\frac{(m-\ell)\zeta}
{cCp^r}\right)\mathrm{d}\zeta.
\end{align*}
\begin{remark}\label{remark1}
The other terms are lower-order terms, which can be treated similarly. 
The same method applies to the remaining sums and 
yields better bounds due to their shorter lengths 
(cf.~\cite[Section 3]{HX} and~\cite[Section 3]{WZ}).
\end{remark}
Interchanging the order of integration and summations, we get
\begin{align}\label{beforeVoronoi}
\widetilde{S}(N)=\frac{1}{C}\sum_{\substack{1\leq c\leq C \\ (c,p)=1}}
\frac{1}{cp^r}\;\sideset{}{^*}\sum_{a\bmod cp^r}
\int_\mathbb{R}g(c,\zeta)S_g(N, a, c, \zeta)S_f(N, a, c, \zeta)\mathrm{d}\zeta,
\end{align}
where
\begin{align}\label{l-sum}
S_g(N, a, c, \zeta):=\sum_{\ell\geq1}\lambda_g(\ell)e\left(-\frac{a\ell}{cp^r}\right)
W\left(\frac{\ell}{N}\right)e\left(-\frac{\ell\zeta}{cCp^r}\right),
\end{align}
and
\begin{align}\label{m-sum}
S_f(N, a, c, \zeta):=\sum_{m\geq1}\lambda_f(m)\chi(m)e\left(\frac{am}{cp^r}\right)
 V\left(\frac{m}{N}\right)e\left(\frac{m\zeta}{cCp^r}\right).
\end{align}

\subsection{Applying Voronoi summation formulas}

In this subsection, we shall apply Voronoi summation formulas
to the $\ell$-and $m$-sums in \eqref{beforeVoronoi}.

We first consider the sum over $\ell$. We apply
Lemma~\ref{voronoiGL2-Maass} with
$$h_1(y):=W(y)e\left(-\frac{ yN\zeta}{cCp^r}\right)$$ to transform the
$\ell$-sum in~\eqref{l-sum} into
\begin{align}\label{$n$-sum after GL2 Voronoi}
S_g(N, a, c, \zeta)=\frac{N}{cp^r}\sum_{\pm}\sum_{\ell \geq 1}\lambda_g(\ell)
e\left(\pm\frac{\ell \overline{a}}{cp^r}\right)
\Psi_{h_1}^{\pm}\left(\frac{\ell N}{c^2p^{2r}}\right),
\end{align}
where by~\eqref{intgeral transform-2},
$$\Psi_{h_1}^{\pm}(x)= \int_0^\infty W(y)e\left(-\frac{ yN\zeta}{cCp^r}\right)
 \mathcal{J}_f^{\pm}(4\pi\sqrt{xy})\dd y.$$
For $x\gg N^{\varepsilon}$, by Lemma~\ref{voronoiGL2-Maass-asymptotic},
we have $\Psi_{h_1}^-(x)\ll N^{-A}$ and the evaluation of 
$\Psi_{h_1}^{+}(x)$ is reduced to considering the integral
\begin{align*}
\Psi_{0}(x, \zeta, c)=x^{-1/4} \int_0^\infty y^{-1/4} W(y)
e\left(\frac{\zeta N y}{cCp^r}\pm 2 \sqrt{xy}\right)\mathrm{d}y.
\end{align*}
Let $$\varrho(y)=\frac{\zeta N y}{cCp^r}\pm 2 \sqrt{xy}.$$ Then we have
$$
\varrho'(y)=\frac{\zeta N }{cCp^r}\pm \sqrt{\frac{x}{y}}, \qquad
\varrho''(y)=\mp \frac{1}{2y}\sqrt{\frac{x}{y}}.
$$
Then repeated integration by parts shows that
the above integral is negligibly small unless we take the $-$ sign and 
$x\asymp (\zeta N)^2/(cCp^r)^2$.
More precisely, by a stationary phase argument,
\begin{align*}
\Psi_0(x,\zeta,c)=x^{-1/2}W_{\natural}\left(\frac{c^2C^2p^{2r}x}{\zeta^2N^2}\right)
e\left(-\frac{cCp^r x}{\zeta N}\right)+O_A(N^{-A})
\end{align*}
for any $A>0$.
Then for $x=\ell N/c^2p^{2r}$, one sees that
the $\ell$-sums in~\eqref{$n$-sum after GL2 Voronoi} can be truncated at
$\ell\leq N^{1+\varepsilon}/C^2 \asymp p^{r+\varepsilon}$, at the cost of a negligible error.
On the other hand, for $x\ll N^{\varepsilon}$, the above restriction on $\ell$ holds trivially.
So in any case, up to a negligible error, we can restricted the sum over $\ell$ to
$\ell \leq p^{r+\varepsilon}$.
Then we arrive at
\begin{equation*}
S_g(N, a, c, \zeta)=N^{1/2}\sum_{\pm}\sum_{\ell \leq p^{r+\varepsilon}}
\frac{\lambda_g(\ell)}{\ell^{1/2}}
e\left(\pm \frac{\ell \overline{a}}{cp^r}\right)
\mathfrak{J}_g^{\pm}(\ell, c, \zeta)+O(N^{-A}),
\end{equation*}
where for $\ell N/c^2p^{2r}\gg N^{\varepsilon}$,
\begin{align}\label{l large}
\mathfrak{J}_g^{\pm}(\ell, c, \zeta)=\left(\frac{\ell N}{c^{2}p^{2r}}\right)^{1/4} \int_0^\infty y^{-1/4}W(y)
e\left(-\frac{\zeta N y}{cCp^r}\pm 2 \frac{\sqrt{\ell Ny}}{cp^r}\right)\mathrm{d}y,
\end{align}
and for $\ell N/c^2p^{2r} \ll N^{\varepsilon}$,
\begin{align}\label{l small}
\mathfrak{J}_g^{\pm}(\ell, c, \zeta)=\frac{\sqrt{\ell N}}{cp^r}
\int_0^{\infty}W(y)e\left(-\frac{\zeta N y}{cCp^r}\right)
\mathcal{J}_g^{\pm}\left(\frac{4\pi\sqrt{\ell Ny}}{cp^r}\right)\mathrm{d}y.
\end{align}

Then we consider the sum over $m$. 
To prepare for the application of the Voronoi summation formula to the $m$-sum in~\eqref{m-sum},
we use the Fourier expansion of $\chi$ in terms of additive characters (see (3.12) in~\cite{IK}),
$$\chi(m)=
\frac{1}{\tau(\overline{\chi})}
\sum_{b \bmod p^n}\overline{\chi}(b)
e\left(\frac{mb}{p^n}\right),$$
one has
\begin{equation*}
S_f(N, a, c, \zeta)=\frac{1}{\tau(\overline{\chi})}\sum_{b \bmod p^n}
\overline{\chi}(b)\sum_{m\geq1}\lambda_f(m)
e\left(\frac{ap^{n-r}+bc}{cp^n}m\right)V\left(\frac{m}{N}\right)
e\left(\frac{m\zeta}{cCp^r}\right).
\end{equation*}
Note that $(ap^{n-r}+bc,c)=1$ and $(ap^{n-r}+bc,p)=(bc,p)=1$. 
Thus $(ap^{n-r}+bc, cp^n)=1$. 
Then we apply Lemma~\ref{voronoiGL2-Maass} with
$$h_2(y):=V(y)e\left(\frac{yN\zeta}{cCp^r}\right)$$ to transform the
$m$-sum in~\eqref{m-sum} into
\begin{equation*}\label{$m$-sum after GL2 Voronoi}
S_f(N, a, c, \zeta)=\frac{N}{\tau(\overline{\chi})cp^n}\sum_{\pm}\sum_{b \bmod p^n}
\overline{\chi}(b)\sum_{m \geq 1}\lambda_f(m)
e\left(\mp\frac{\overline{ap^{n -r}+bc}}{cp^n}m\right)
\Psi_{h_2}^{\pm}\left(\frac{mN}{c^2p^{2n}}\right),
\end{equation*}
where by~\eqref{intgeral transform-1},
\begin{equation*}
\Psi_{h_2}^{\pm}\left(\frac{mN}{c^2p^{2n}}\right)
=\frac{1}{2\pi i}\int_{(\sigma)}\left(\frac{mN}{c^2p^{2n}}\right)^{-s}
\gamma_f^{\pm}(s)
V^{\dag}\left(\frac{N\zeta}{cCp^r},-s\right)\mathrm{d}s
\end{equation*}
with
\begin{equation*}
V^{\dag}(\xi,s) =\int_{0}^{\infty}V(y)e(\xi y)y^{s-1} \mathrm{d} y.
\end{equation*}
By Stirling's formula, for $\sigma\geq -1/2$,
\begin{equation}\label{A bound}
\gamma_f^{\pm}(\sigma+i\tau)\ll_{f,\sigma}(1+|\tau|)^{2\left(\sigma+1/2\right)}.
\end{equation}
By \cite[Lemma 5]{Mun1}, we have
\begin{equation*}
\begin{split}
V^{\dag}\left(\frac{N\zeta}{cCp^r},-s\right)\ll_j&
\min \left\{ \left(\frac{1+|\mathrm{Im}(s)|}{N|\zeta|/cCp^r} \right)^j ,
\left(\frac{1+N|\zeta|/cCp^r}{|\mathrm{Im}(s)|} \right)^j \right\}\\
\ll_j& \min \left\{1, \left(\frac{N^{1+\varepsilon}}{cCp^r|\mathrm{Im}(s)|} \right)^j \right\}.
\end{split}
\end{equation*}
This together with~\eqref{A bound} implies that
\begin{align*}
\Psi_{h_2}^{\pm}\left(\frac{mN}{c^2p^{2n}}\right)
&\ll_j \left(\frac{mN}{c^2p^{2n}}\right)^{-\sigma}
\int_{\mathbb{R}}
(1+|\tau|)^{2(\sigma+1/2)}\min\left\{1, \left(\frac{N^{1+\varepsilon}}
{cCp^r|\mathrm{Im}(s)|} \right)^j \right\}\mathrm{d}\tau\\
&\ll\left(\frac{N^{1+\varepsilon}}{cCp^r}\right)^{2}\left(\frac{p^{2r}C^2m}
{N^{1+\varepsilon}p^{2n}}\right)^{-\sigma},
\end{align*}
by choosing $j=2\sigma+2$ and noting that $|\zeta|\leq N^{\varepsilon}$.
Then taking $\sigma$ sufficiently large, we see the
$m$-sum in~\eqref{$n$-sum after GL2 Voronoi} can be truncated at
$m\leq N^{1+\varepsilon}p^{2(n-r)}/C^2 \asymp p^{2n-r+\varepsilon}$.
After this truncation, we shift the contour to $\sigma=-1/2$ to obtain
\begin{align*}
S_f(N, a, c, \zeta)=\frac{N^{1/2}}{\tau(\overline{\chi})}&\sum_{\pm}\sum_{b \bmod p^n}
\overline{\chi}(b)\sum_{m\leq p^{2n-r+\varepsilon}}\frac{\lambda_f(m)}{m^{1/2}}\\
&\times e\left(\mp\frac{\overline{ap^{n-r}+bc}}{cp^n}m\right)
\mathfrak{J}_f^{\pm}(m, c, \zeta)+O(N^{-A}),
\end{align*}
where
\begin{equation}\label{integral J}
\mathfrak{J}_f^{\pm}(m, c, \zeta)
=\frac{1}{2\pi i}\int_{(\sigma)}\left(\frac{Nm}{c^2p^{2n}}\right)^{-i\tau}
\gamma_f^{\pm}\left(-\frac{1}{2}+i\tau\right)
V^{\dag}\left(\frac{N\zeta}{cCp^r},\frac{1}{2}-i\tau\right)\mathrm{d}\tau.
\end{equation}
By assembling the above results and interchanging the order of summation and integration, we obtain
\begin{align}\label{main case}
\widetilde{S}(N)=\;&\sum_{\pm}\sum_{\pm}\frac{N}{Cp^r}
\sum_{\substack{1\leq c\leq C \\ (c, p)=1}}\frac{1}{c}
\sum_{m\leq p^{2n-r+\varepsilon}}\frac{\lambda_f(m)}{m^{1/2}}
\sum_{\ell\leq p^{r+\epsilon}}\frac{\lambda_g(\ell)}{\ell^{1/2}}\nonumber\\
&\times H^{\pm, \pm}(m, \ell, c)K(\mp m, \mp \ell, c)+O_A\left(N^{-A}\right),
\end{align}
where
\begin{equation}\label{integral JJ}
H^{\pm, \pm}(m, \ell, c)=\int_{\mathbb{R}}g(c,\zeta)\mathfrak{J}_f^{\pm}(m, c, \zeta)
\mathfrak{J}_g^{\pm}(\ell, c, \zeta)\mathrm{d}\zeta,
\end{equation}
and
\begin{align*}
K(m,\ell,c)=
\frac{1}{\tau(\overline{\chi})}\;\sum_{b \bmod p^n}\overline{\chi}(b)
\;\sideset{}{^*}\sum_{a\bmod{cp^r}}
e\left(\frac{m\overline{ap^{n-r}+bc}}{cp^n}\right)e\left(-\frac{\ell\overline{a}}{cp^r}\right).
\end{align*}

Before further analysis, we make a computation of the character sums 
$K(m,\ell,c)$. Using the well-known reciprocity formula, we obtain
$$
e\left(\frac{m\overline{(ap^{n-r}+bc)}}{cp^n}\right)=
e\left(\frac{m\overline{c}\overline{(ap^{n-r}+bc)}}{p^n}\right)
e\left(\frac{m\overline{ap^{2n-r}}}{c}\right)
$$
and
$$
e\left(\frac{\ell\overline{a}}{cp^r}\right)=
e\left(\frac{\overline{c}\ell\overline{a}}{p^r}\right)
e\left(\frac{\overline{p^r}\ell\overline{a}}{c}\right).
$$
Therefore, $K(m, \ell, c)$ splits into a product of two 
sums of $c$ and $p^r$ respectively. More precisely, we get
\begin{align}\label{K sum}
K(m,\ell,c)=S\big(0, m\overline{p^{2n-r}}-\ell\overline{p^r}; c\big)G(m,\ell,c),
\end{align}
where $S(0,u;c)$ is the Ramanujan sum and
\begin{align*} 
G(m,\ell,c)=\frac{1}{\tau(\overline{\chi})}\sum_{b \bmod p^n} \overline{\chi}(b)
\;\sideset{}{^*}\sum_{a\bmod{p^r}}e\left(\frac{m\overline{c}\overline{(ap^{n-r}+bc)}}{p^r}\right)
e\left(-\frac{\overline{c}\ell\overline{a}}{p^r}\right).
\end{align*}
Plugging~\eqref{K sum} into~\eqref{main case} and further breaking the
$\ell, m$-sums into dyadic segments in $\widetilde{S}(N)$, we obtain
$$
\widetilde{S}(N)\ll \frac{N^{1/2+\varepsilon}}{p^{r/2}}
\sup_{1\ll M\ll p^{2n-r+\varepsilon} }
\sup_{1\ll L\ll p^{r+\varepsilon}}
\sum_{1\leq c\leq C\atop (c,p)=1}
\sum_{d|c}\frac{d}{c}|\widetilde{S}_d^{\pm, \pm }(N,L,M)|,
$$
where
\begin{equation}\label{Sd}
\widetilde{S}_d^{\pm, \pm }(N,L,M)=\sum_{m\sim M}\frac{\lambda_f(m)}{m^{1/2}}
\sum_{\substack{\ell\sim L \\ \ell \equiv mp^{\lambda}\overline{p^{2\kappa-r}}\bmod d}}
\frac{\lambda_g(\ell)}{\ell^{1/2}}H^{\pm, \pm}(m, \ell, c)G(\mp m, \mp \ell, c).
\end{equation}
Here we have used the following identity for the Ramanujan sum
\begin{align*}
S(0,u;c)=\sum_{d\mid (u,c)}d\mu\left(\frac{c}{d}\right).
\end{align*} 

\section{Cauchy and Poisson}

Applying the Cauchy--Schwartz inequality to $m$-sum in \eqref{Sd} and using
the Rankin--Selberg estimate~\eqref{GL2: Rankin Selberg}, one sees that
$$
\widetilde{S}_d^{\pm, \pm}(N, L, M) \ll \mathcal{L}(L,M)^{1/2},
$$
where
\begin{align}\label{L-sum}
\mathcal{L}(L,M)=\sum\limits_{m} U\left(\frac{m}{M}\right)
\Bigg|\sum_{\substack{\ell\sim L \\ \ell \equiv mp^r\overline{p^{2n-r}}\bmod d}}
\frac{\lambda_g(\ell)}{\ell^{1/2}}H^{\pm, \pm}(m, \ell, c)G(\mp m, \mp \ell, c)
\Bigg|^2,
\end{align}
with $U$ is a smooth function supported in $[1/2,5/2]$, $U(x)=1$
for $x\in [1,2]$ and $U^{(j)}(x)\ll_j 1$.
Opening the square and switching the order of summations in~\eqref{L-sum}, we get
\begin{align}\label{L estimates}
\mathcal{L}(L,M)=
\sum_{\ell_1\sim L}\frac{\lambda_g(\ell_1)}{\ell_1^{1/2}}
\sum_{\substack{\ell_2\sim L \\ \ell_2\equiv \ell_1\bmod d}}
\frac{\overline{\lambda_g(\ell_2)}}{\ell_2^{1/2}}\mathcal{T}(L,M).
\end{align}
where
\begin{align*}
\mathcal{T}(L,M)=&
\sum_{\substack{m\in \mathbb{Z}\\ m\equiv \ell_1p^{2n-r}
\overline{p^r}\bmod d}}U\left(\frac{m}{M}\right)
H^{\pm, \pm}(m, \ell_1, c)\overline{H^{\pm, \pm}(m, \ell_2, c)}\\
&\times 
G(\mp m, \mp \ell_1, c)\overline{G(\mp m, \mp \ell_2, c)}.
\end{align*}
We break the $m$-sum into congruence classes modulo $dp^n$ and apply
the Poisson summation formula to the sum over $m$. 
It is therefore sufficient to consider the following sum
\begin{equation*}
\mathcal{T}(L,M)=\frac{M}{dp^n}
\sum_{\tilde{m}\in \mathbb{Z}}|\mathcal{K}(\tilde{m})||\mathcal{I}(\tilde{m})|,
\end{equation*}
where the character sum $\mathcal{K}(\tilde{m})
:=\mathcal{K}(\tilde{m}; \mp \ell_1, \mp \ell_2,c)$ 
is given by
\begin{equation}\label{character sums}
\mathcal{K}(\tilde{m})=
\sum_{\substack{u\bmod dp^n \\ u \equiv \ell_1p^{2n-r}
\overline{p^r}\bmod d}} G(\mp u, \mp \ell_1, c)
\overline{G(\mp u, \mp \ell_2, c)}
e\left(\frac{\tilde{m} u}{dp^n}\right),
\end{equation}
and the integral $\mathcal{I}(\tilde{m})=\mathcal{I}(\tilde{m}; \ell_1, \ell_2,c)$ is given by
\begin{equation}\label{integral-H}
 \mathcal{I}(\tilde{m})=\int_{\mathbb{R}}
U\left(\xi\right)H^{\pm, \pm}\left(M\xi,\ell_1,c\right)
\overline{H^{\pm, \pm}\left(M\xi,\ell_2,c\right)}\,
e\left(-\frac{\tilde{m} M \xi}{dp^n}\right)\mathrm{d}\xi.
\end{equation}

We need to analyze the integral $\mathcal{I}(\tilde{m})$ and the character sum $\mathcal{K}(\tilde{m})$. 
The integral $\mathcal{I}(\tilde{m})$ will be computed 
by the two dimensional (archimedean) second derivative test;
while the character sum $\mathcal{K}(\tilde{m})$ will be
evaluated by an involved non-archimedean stationary phase computation 
and the $p$-adic second derivative test.
\subsection{Estimation of the integral}
In this subsection, we have the following estimate for $\mathcal{I}(\tilde{m})$.
\begin{lemma}\label{integral:lemma}
Let $\mathcal{I}(\tilde{m})$ be defined as in \eqref{integral-H}. Then,
one has the following estimates.
\begin{enumerate}
\item
If $\tilde{m} \gg \frac{N^\varepsilon Cdp^n}{cM}$, we have $ \mathcal{I}(\tilde{m}) \ll N^{-A}$.

\item
If $\tilde{m} \ll \frac{N^\varepsilon Cdp^n}{cM}$, we have 
$$ \mathcal{I}(\tilde{m}) \ll  \frac{N^\varepsilon c}{C}.$$

\end{enumerate}
\end{lemma}
\begin{proof}
By \eqref{integral JJ}, we have
\begin{equation*}
\xi^j \frac{\partial^j }{\partial \xi^j}H^{\pm, \pm}\left(M\xi,\ell,c\right)=
\int_{\mathbb{R}}g(c,\zeta)\mathfrak{J}_g^{\pm}(\ell, c, \zeta)
\xi^j \frac{\partial^j }{\partial \xi^j}
\mathfrak{J}_f^{\pm}(M\xi, c, \zeta)\mathrm{d}\zeta,
\end{equation*}
where by \eqref{integral J},
\begin{align*}
\xi^j \frac{\partial^j }{\partial \xi^j}
\mathfrak{J}_f^{\pm}(M\xi, c, \zeta)=\;&
\frac{1}{2\pi}\int_{\mathbb{R}}(-i\tau)(-i\tau-1)\cdots (-i\tau-j+1) \\&\times
\left(\frac{NM\xi}{c^2p^{2r}}\right)^{-i\tau}
\gamma_f^{\pm}\left(-\frac{1}{2}+i\tau\right)
V^{\dag}\left(\frac{N\zeta}{cCp^r},\frac{1}{2}-i\tau\right)\mathrm{d}\tau.
\end{align*}
By definition,
\begin{equation*}
V^{\dag}\left(\frac{N\zeta}{cCp^r},\frac{1}{2}-i\tau\right)=
\int_0^{\infty}V(y_2)y_2^{-1/2}e\left( -\frac{\tau}{2\pi}\log y_2
+\frac{N\zeta y_2}{cCp^r} \right)\mathrm{d}y_2.
\end{equation*}
By repeated integration by parts implies that the above integral
is negligible unless $ |\tau|\asymp |\xi|N/(cCp^r):=X$.
Moreover, by the second derivative test for exponential integrals
in~\cite[Lemma 5.1.3]{Hux2},
\begin{equation}\label{V-dag11}
V^{\dag}\left(\frac{N\zeta}{cCp^r},\frac{1}{2}-i\tau\right)\ll (1+|\tau|)^{-1/2}.
\end{equation}
Consequently,
\begin{equation*}
\begin{split}
\xi^j \frac{\partial^j }{\partial \xi^j}
\mathfrak{J}_f^{\pm}(M\xi, c, \zeta)=\;&\frac{1}{2\pi}
\int_{\mathbb{R}}\omega\Big(\frac{|\tau|}{X}\Big)
(-i\tau)(-i\tau-1)\cdots (-i\tau-j+1) \\&\times
\left(\frac{NM\xi}{c^2p^{2r}}\right)^{-i\tau}
\gamma_f^{\pm}\left(-\frac{1}{2}+i\tau\right)
V^{\dag}\left(\frac{N\zeta}{cCp^r},
\frac{1}{2}-i\tau\right)\mathrm{d}\tau+O(N^{-A}).
\end{split}
\end{equation*}
where $\omega(x)\in C_c^{\infty}(0,\infty)$
satisfies $\omega^{(j)}(x)\ll_j 1$ for any integer $j\geq 0$.
Hence
\begin{align}\label{H++}
\xi^j \frac{\partial^j }{\partial \xi^j}
H^{\pm, \pm}\left(M\xi,\ell,c\right)=\;&\frac{1}{2\pi}
\int_{\mathbb{R}}g(c, \zeta)\int_{\mathbb{R}}\omega\Big(\frac{|\tau|}{X}\Big)(-i\tau)
(-i\tau-1)\cdots (-i\tau-j+1) \notag\\
&\times\int_0^{\infty}V(y_2)y_2^{-1/2}e\left( -\frac{\tau}{2\pi}\log y_2
+\frac{N\zeta y_2}{cCp^r} \right)\mathrm{d}y_2 \notag \\&\times
\left(\frac{NM\xi}{c^2p^{2r}}\right)^{-i\tau}\gamma_f^{\pm}\left(-\frac{1}{2}+i\tau\right)
\mathfrak{J}_g^{\pm}(\ell, c, \zeta)\mathrm{d}\tau\mathrm{d}\zeta+O(N^{-A}),
\end{align}
where by~\eqref{l large} for $\ell N/c^2p^{2\lambda}\gg N^{\varepsilon}$,
\begin{align*}
\mathfrak{J}_g^{\pm}(\ell, c, \zeta)=\frac{(\ell N)^{1/4}}{c^{1/2}p^{r/2}} \int_0^\infty y_1^{-1/4}
W(y_1)e\left(\frac{\zeta N y_1}{cCp^r}\pm 2 \frac{\sqrt{\ell Ny_1}}{cp^r}\right)\mathrm{d}y_1,
\end{align*}
and by~\eqref{l small} for $\ell N/c^2p^{2r} \ll N^{\varepsilon}$,
\begin{align}\label{I small}
\mathfrak{J}_g^{\pm}(\ell, c, \zeta)=\frac{\sqrt{\ell N}}{cp^r}
\int_0^{\infty}W(y_1)e\left(\frac{\zeta N y_1}{cCp^r}\right)
\mathcal{J}_g^{\pm}\left(\frac{4\pi\sqrt{\ell Ny_1}}{cp^r}\right)\mathrm{d}y_1.
\end{align}
It is easily verified that the integral
\begin{align*}
\int_{\mathbb{R}} g(c,\zeta) e\left( \frac{\zeta N(y_2-y_1)}{cCp^r} \right) \mathrm{d}\zeta
\end{align*}
is negligible unless \( |y_1-y_2| \ll N^{\varepsilon}cCp^r/N \ll N^{\varepsilon}c/C \), 
where \( C = \sqrt{N/p^r} \) (via \eqref{g rapid decay}).  
Using Stirling's formula, we express
\begin{equation*}
\gamma_f^{\pm}\left(-\tfrac{1}{2}+i\tau\right) = \left( \frac{|\tau|}{e\pi} \right)^{2i\tau} \Upsilon_{\pm}(\tau), \quad \Upsilon_{\pm}^{j}(\tau) \ll |\tau|^{-j},
\end{equation*}
for any integer $j \geq 1$.
Thus for $\ell N/c^2p^{2r} \gg N^{\varepsilon}$, applying the substitution 
$y=y_2-y_1$, we rewrite the integral~\eqref{H++} as
\begin{align*}
\xi^j \frac{\partial^j }{\partial \xi^j}H^{\pm, \pm}\left(M\xi,\ell,c\right)=\;&
\frac{(\ell N)^{1/4}}{2\pi c^{1/2}p^{r/2}}\int_{\mathbb{R}}g(c,\zeta)
\int_{|y|\ll N^{\varepsilon}c/C}e\left(\frac{\zeta N y}{cCp^r} \right)\\&
 \times \int_{\mathbb{R}}\int_0^\infty V_0(y_1,\tau;y)
e\left(F(y_1,\tau;y)\right)\mathrm{d}y_1\mathrm{d}\tau\mathrm{d}y\mathrm{d}\zeta.
\end{align*}
where
\begin{align*}
V_0(y_1,\tau;y) = &\; \omega\Big(\frac{|\tau|}{X}\Big)y_1^{-1/4}(y_1+y)^{-1/2}V(y_1)W(y_1+y)\\
&\times(-i\tau)(-i\tau-1)\cdots (-i\tau-j+1) \Upsilon_{\pm}(\tau),
\end{align*}
and
$$
F(y_1,\tau;y)=-\frac{\tau}{2\pi}\log\frac{NM\xi}{c^2p^{2r}}+\frac{\tau}{\pi}\log\frac{|\tau|}{e\pi}
-\frac{\tau}{2\pi}\log (y_1+y)\pm\frac{2\sqrt{\ell Ny_1}}{cp^r}.
$$
Note that the phase function derivatives satisfy
\begin{align*}
\frac{\partial F(y_1,\tau;y)}{\partial y_1}
     &=-\frac{\tau}{2\pi (y_1+y)}\pm \frac{\sqrt{\ell N}}{cp^r}y_1^{-1/2},\\
\frac{\partial F(y_1,\tau;y)}{\partial \tau}
     &=-\frac{1}{2\pi}\log\frac{NM\xi}{c^2p^{2r}}+\frac{1}{\pi}\log\frac{|\tau|}{\pi}
     -\frac{1}{2\pi}\log (y_1+y).
\end{align*}
Thus we have
\begin{align*}
\frac{\partial^2 F(y_1,\tau;y)}{\partial y_1^2}
     &=\frac{\tau}{2\pi (y_1+y)^2}\mp \frac{1}{2}\frac{\sqrt{\ell N}}{cp^r}y_1^{-3/2}\gg |\tau|,\\
\frac{\partial^2 F(y_1,\tau;y)}{\partial \tau^2}
     &=\frac{1}{\pi \tau}\gg |\tau|^{-1},\\
\frac{\partial^2 F(y_1,\tau;y)}{\partial y_1\partial \tau}&= -\frac{1}{2\pi (y_1+y)}.
\end{align*}
This implies that
\begin{align*}
\left|\mathrm{det}F''\right|=\left|\frac{\partial^2 F}{\partial y_1^2}
\frac{\partial^2 F}{\partial \tau^2}
-\left(\frac{\partial^2 F}{\partial y_1\partial \tau}\right)^2\right|\gg 1.
\end{align*}

By applying the two dimensional second derivative test in
\cite[Lemma 5.1.3]{Hux2} with
$\rho_1=|\tau|$, $\rho_2=|\tau|^{-1}$,
we have
\begin{align*}
\int_{\mathbb{R}}\int_0^\infty V_0(y_1,\tau;y)
e\left(F(y_1,\tau;y)\right)\mathrm{d}y_1\mathrm{d}\tau
\ll X^{j}\ll  \left(\frac{|\xi|N}{cCp^r}\right)^{j}
\ll \left(\frac{C}{c}\right)^{j}
\end{align*}
recalling $ C= \sqrt{N/p^r}$. It follows that
\begin{align*}
\xi^j \frac{\partial^j }{\partial \xi^j}H^{\pm, \pm}\left(M\xi,\ell,c\right)\ll
N^\varepsilon \left(\frac{C}{c}\right)^{j-1/2}.
\end{align*}
This also holds when $\ell N/c^2p^{2r} \ll N^{\varepsilon}$. 
In fact, by~\eqref{H++} and~\eqref{I small}, we have
\begin{align*}
\xi^j \frac{\partial^j }{\partial \xi^j}H^{\pm, \pm}\left(M\xi, \ell ,c\right)
&=\frac{\sqrt{\ell N}}{2\pi cp^r}\int_0^{\infty}V(y_1)
\mathcal{J}_g^{\pm}\left(\frac{4\pi\sqrt{\ell Ny_1}}{cp^r}\right)
\nonumber\\&\times\int_{\mathbb{R}}\omega\Big(\frac{|\tau|}{X}\Big)
  (-i\tau)(-i\tau-1)\cdots (-i\tau-j+1)
\left(\frac{NM\xi}{c^2p^{2r}}\right)^{-i\tau}
\gamma_f^{\pm}\left(-\frac{1}{2}+i\tau\right)\nonumber\\&
\times\int_{\mathbb{R}}g(c,\zeta)e\left(\frac{N\zeta}{cCp^r}\right)
V^{\dag}\left(\frac{N\zeta}{cCp^r},\frac{1}{2}-i\tau\right)
\mathrm{d}y_1 \mathrm{d}\tau \mathrm{d}\zeta+O(N^{-A}).
\end{align*}
As before, by considering the $\zeta$-integral and integrating by parts, we find that
the above integral is negligible unless 
$|y_1-y_2|\ll N^{\varepsilon}cCp^r/N\ll N^{\varepsilon}c/C$.
Furthermore, by Remark~\ref{JK remark}, we have
$$\mathcal{J}_g^{\pm}\left(\frac{4\pi\sqrt{\ell Ny_1}}{cp^r}\right)
\ll \left(\frac{\ell N }{c^2p^{2r}}\right)^{-1/4}.$$
Combining these estimates with~\eqref{V-dag11}, we obtain
\begin{align*}
\xi^j \frac{\partial^j }{\partial \xi^j}H^{\pm, \pm}\left(M\xi,\ell,c\right)
\ll  \left(\frac{\ell N }{c^2p^{2r}}\right)^{1/4}\frac{cN^{\varepsilon}}{C}
X^{j+1/2}\ll N^{\varepsilon}\left(\frac{C}{c}\right)^{j-1/2}.
\end{align*}
Hence by applying integration by parts repeatedly on the $\xi$-integral in \eqref{integral-H} and evaluating the
resulting $\xi$-integral trivially, we find that
$$
\mathcal{I}(\tilde{m}) \ll_j N^\varepsilon\; \frac{c}{C}
\left(\frac{\tilde{m} M c}{Cdp^n}\right)^{-j}
$$
for any integer $ j\geq 0$. Therefore, $\mathcal{I}(\tilde{m}) $ is negligible 
unless $$\tilde{m} \leq \frac{N^\varepsilon Cdp^n}{cM}. $$
Moreover, by taking $ j=0$, one has
$$
\mathcal{I}(\tilde{m}) \ll_j \frac{N^\varepsilon c}{C}.
$$
This completes the proof of the Lemma \ref{integral:lemma}.
\end{proof}

\subsection{Evaluation of the character sum}
In this subsection we estimate the character sums in~\eqref{character sums}.
We write $u=d \overline{d}u_{1}+p^n \overline{p^n}u_{2}$ , 
with $u_1\bmod p^n$ and $u_2\bmod d$. Then we obtain
\begin{equation*}
\mathcal{K}(\tilde{m})=
\sum_{u_1\bmod p^n}G(u_1, \mp \ell_1,c)\overline{G(u_1, \mp \ell_2,c)}
e\left(\pm \frac{\tilde{m} \overline{d} u_1}{p^n}\right)
\;\sideset{}{^*}\sum_{u_2\bmod d \atop u_2\equiv \ell_1p^{2n-r}
\overline{p^r}\bmod d}e\left(\frac{\tilde{m} \overline{p^n} u_2}{d}\right).
\end{equation*}
We therefore arrive at the expression
\begin{equation*}
\mathcal{K}(\tilde{m})=e\left(\frac{\tilde{m}\ell_1p^{n-r}\overline{p^r}}{d}\right)
\mathcal{C}(\pm \tilde{m}, \mp \ell_1, \mp \ell_2),
\end{equation*}
where
\begin{equation}\label{C(m)}
\mathcal{C}(\tilde{m}, \ell_1, \ell_2)=
\sum_{u\bmod p^n}G(u,\ell_1,c)\overline{G(u,\ell_2,c)}
e\left(\frac{\tilde{m} \overline{d} u}{p^n}\right).
\end{equation}

So to estimate $\mathcal{K}(\tilde{m})$, we only need to evaluate $\mathcal{C}(\tilde{m}, \ell_1, \ell_2)$. To evaluate such a sum further, we will need to make a few elementary transformations. Again, using the Fourier expansion of $\chi$ in terms of additive characters, we have 
\begin{align*} 
G(m,\ell,c)
=\frac{1}{q}\sum_{u \bmod p^n}\;\sideset{}{^*}\sum_{b \bmod p^n}
\;\sideset{}{^*}\sum_{a\bmod{p^r}}\chi(u)e\left(\frac{bu}{p^n}+\frac{m\overline{c}
\overline{(ap^{n-r}+bc)}}{p^n}\right)e\left(-\frac{\overline{c}\ell\overline{a}}{p^r}\right).
\end{align*}
On changing the variable $ap^{n-r}+bc=h$, we get
\begin{align}\label{character sums0}
G(m,\ell,c)=&\;\frac{1}{q}\sum_{u \bmod p^n}\;\sideset{}{^*}\sum_{h \bmod p^n}
\;\sideset{}{^*}\sum_{a\bmod{p^r}}\chi(u)e\left(\frac{u\overline{c}(h-ap^{n-r})}{p^n}
+\frac{m\overline{c}\overline{h}}{p^n}\right)
e\left(-\frac{\overline{c}\ell\overline{a}}{p^r}\right)\nonumber\\
=&\;\frac{1}{q}\sum_{u \bmod p^n} \chi (u) 
S\big(u\overline{c}, m\overline{c}; p^n\big)
S\big(u \overline{c}, \ell\overline{c}; p^r\big).
\end{align}
\begin{remark}
It is precisely due to this transformation 
that we can directly apply the $p$-adic method of stationary phase 
to estimate the character sum $G(m,\ell,c)$,
which is also why our results outperform Sun's~\cite{S}.
This trick is similar in spirit to that of Lin, Michel, 
and Sawin~\cite[Section 4.2]{LMS}, who used high-level 
algebraic geometry to estimate the algebraic exponential 
sums that arise in their calculations.
\end{remark}
Using this transformation, we can get 
\begin{align*}
\mathcal{C}(\tilde{m}, \ell_1, \ell_2)=\frac{1}{q^2}&\sum_{\beta\bmod p^n} 
\sum_{u_1 \bmod p^r}\chi(u_1)S\big(u_1\overline{c}, \beta\overline{c}; p^n\big)
S\big(u_1 \overline{c}, \ell_1\overline{c}; p^r\big)\\
&\times\sum_{u_2 \bmod p^r} \overline{\chi} (u_2)
S\big(u_2\overline{c}, \beta\overline{c}; p^n\big)
S\big(u_2 \overline{c}, \ell_2\overline{c}; p^r\big)
e\left(\frac{\tilde{m}\overline{d}\beta}{p^n}\right).
\end{align*}
We now provide bounds for the character sums $\mathcal{C}(\tilde{m}, \ell_1, \ell_2)$ 
for $\tilde{m}\in \mathbb{Z}$ with different methods on whether $\tilde{m}=0$ or not.

For the case $\tilde{m}=0$, an estimate of $\mathcal{C}(0, \ell_1, \ell_2)$ can 
be obtained in an elementary manner by reducing the problem 
to the Ramanujan sum. This yields the following lemma.
\begin{lemma}\label{zero}
Let $\mathcal{C}(\tilde{m}, \ell_1, \ell_2)$ be defined as in~\eqref{C(m)}. 
We have
$$\mathcal{C}(0, \ell_1, \ell_2) = p^{n+r}
\delta (\ell_1\equiv \ell_2\bmod p^r)
-p^{n+r-1}\delta (\ell_1\equiv \ell_2\bmod p^{r-1}).$$
\end{lemma}
\begin{proof}
Let $\tilde{m}=0$, we have
\begin{align*}
\mathcal{C}(0, \ell_1, \ell_2)=\frac{1}{q^2}&\sum_{\beta\bmod p^n} 
\sum_{u_1 \bmod p^n}\chi(u_1)S\big(u_1\overline{c}, \beta\overline{c}; p^n\big)
S\big(u_1 \overline{c}, \ell_1\overline{c}; p^r\big)\\
&\times\sum_{u_2 \bmod p^n} \overline{\chi} (u_2)  
S\big(u_2\overline{c}, \beta\overline{c}; p^n\big)
S\big(u_2 \overline{c}, \ell_2\overline{c}; p^r\big).
\end{align*}
Opening the Kloosterman sums and executing the sum over $\beta$, we arrive at
\begin{align*}
\mathcal{C}(0, \ell_1, \ell_2)=\frac{1}{q}&\;\sum_{u_1 \bmod p^n}\sum_{u_2 \bmod p^n} \chi (u_1)\overline{\chi}(u_2)
\;\sideset{}{^*}\sum_{\gamma_1\bmod p^r}
e\left(\frac{\gamma_1u_1 \overline{c}+\overline{\gamma_1}\ell_1 \overline{c}}{p^r}\right)
\\&\times \;\sideset{}{^*}\sum_{\gamma_2\bmod p^r}
e\left(\frac{\gamma_2 u_2 \overline{c}+\overline{\gamma_2}\ell_2 \overline{c}}{p^r}\right)
\;\sideset{}{^*}\sum_{\gamma\bmod p^n}
e\left(\frac{u_1 \overline{c}-u_2 \overline{c}}{p^n}\gamma\right).
\end{align*}
The sum over $\gamma$ equals $S(0,u_1 \overline{c}-u_2 \overline{c}; p^n)$ 
is the Ramanujan sum, then we have
\begin{align}\label{2 terms}
S(0,u_1 \overline{c}-u_2 \overline{c}; p^n)=p^{n}\delta(u_2\equiv u_1\bmod p^n)
-p^{n-1}\delta(u_2\equiv u_1\bmod p^{n-1}).
\end{align}
We first consider the contribution from the second 
term in~\eqref{2 terms} to $\mathcal{C}(0, \ell_1, \ell_2)$.
Note that
\begin{align}\label{second term}
\sum_{\substack{u_2 \bmod p^n\\u_2\equiv u_1\bmod p^{n-1}}}
\overline{\chi}(u_2)e\left(\frac{\gamma_2 u_2  \overline{c}}{p^r}\right)
=e\left(\frac{\gamma_2 u_2  \overline{c}}{p^r}\right)\overline{\chi}(u_1)
\sum_{\nu \bmod p }\chi(1+\overline{u_1}p^{n-1}\nu).
\end{align}
Recall that $\chi$ is a primitive character of modulus $p^n$.
Thus $\chi(1+tp^{n-1})$ is an additive character of modulus $p$.
By Lemma~\ref{Postnikov}, there exists $\alpha \in \mathbb{Z}_p^{\times}$ such that
$\chi(1+tp^{n-1})=e(\alpha t/p)$. Therefore the $\nu$-sum in~\eqref{second term} vanishes and 
the contribution from the second term in~\eqref{2 terms} to $\mathcal{C}(0, \ell_1, \ell_2)$ is zero. 
Plugging the first term in~\eqref{2 terms} into $\mathcal{C}(0, \ell_1, \ell_2)$, we get
\begin{align*}
\mathcal{C}(0, \ell_1, \ell_2)
&=\;\sideset{}{^*}\sum_{\gamma_1\bmod p^r}\;\sideset{}{^*}\sum_{\gamma_2\bmod p^r}
e\left(\frac{\overline{\gamma_1}\ell_1 \overline{c}+\overline{\gamma_2}\ell_2 \overline{c}}{p^r}\right)
 \sum_{u_1 \bmod p^n} e\left(\frac{(\gamma_2+\gamma_1)\overline{c}u_1}{p^r}\right)\\
&=p^n\;\sideset{}{^*}\sum_{\gamma_1\bmod p^r}
e\left(\frac{(\ell_1-\ell_2)\overline{c}\gamma_1}{p^r}\right)\\
&=p^nS(0, \ell_1-\ell_2; p^r).
\end{align*}
Thus the lemma follows.
\end{proof}

We split the proof the case $\tilde{m}\not=0$ in two steps, which are treated in 
Lemmas~\ref{K(m,l,c)} and~\ref{p-adic van der Corput estimates} below.

\subsubsection{$p$-adic stationary phase computation}
Now we need to analyze the character sum $G(m,\ell,c)$ by a non-archimedean stationary phase argument.
The resulting sum is a complete exponential sum, it is exactly evaluated in 
the following lemma by using the $p$-adic method of stationary phase, Lemma~\ref{statphase-lemma},
and exploiting $p$-adic local information; this is also facilitated by an 
explicit evaluation of the Kloosterman sum in Lemma~\ref{kloost-eval}.
\begin{lemma}\label{K(m,l,c)}
Let $G(m,\ell,c)$ be defined as in~\eqref{character sums0}. 
Then $G(m,\ell,c)=0$ unless $mc^{-2}  \in \Z_p^{\times 2}$ and $\ell c^{-2} \in \Z_p^{\times 2}$, 
in which case $G(m,\ell,c)$ can be exactly evaluated as follows. Let 
$\rho_1$ resp. $\rho_2$ be $0$ or $1$ according as $n$ resp. $r$ is even or odd.
For $\sigma_1, \sigma_2 \in \{\pm 1\}$, define $\mathcal{Z}^{\sigma_1, \sigma_2}_{\ell, c}(u)$ 
as for $u\in \mathbb{Z}_p^{\times 2}$, 
\begin{align*}
\mathcal{Z}^{\sigma_1, \sigma_2}_{\ell, c}(u) := 
\epsilon\big(\sigma_1\alpha, p^{\rho_1}\big)
\epsilon\big(\sigma_2\alpha(\overline{u}\ell)_{1/2}, p^{\rho_2}\big)
\overline{\chi}\big((\sigma_1u_{1/2}+\sigma_2\ell_{1/2}p^{n-r})^2\big).
\end{align*}
Then we have
\begin{align*}
G(m,\ell,c) =  p^{r/2}\chi(\alpha^2c^2)
\theta\left(-\frac{2\alpha}{p^n}\right)
\epsilon\big(-\overline{2}\alpha, p^{\rho_1}\big)
\sum_{\sigma_1, \sigma_2\in\{\pm1\}}
\mathcal{Z}^{\sigma_1, \sigma_2}_{\ell,c}(m).
\end{align*}
\end{lemma}
\begin{proof}
We can evaluate the Kloosterman sum by Lemma~\ref{kloost-eval}. We find
that $  S\big(u\overline{c}, m\overline{c}; p^n\big) = 0$ unless 
$umc^{-2}  \in \Z_p^{\times 2}$, and 
$S\big(u \overline{c}, \ell\overline{c}; p^r\big) = 0$
unless $u\ell c^{-2} \in \Z_p^{\times 2}$. If this condition is satisfied,
then
\begin{align*}
G(m,\ell,c)=p^{(r-n)/2}&\sum_{u \bmod p^n}
\chi(u)\sum_{\sigma_1, \sigma_2\in\{\pm1\}}
\epsilon\big(\sigma_1(um)_{1/2}\overline{c}, p^{\rho_1}\big)\\
&\times\epsilon\big(\sigma_2(u\ell)_{1/2}\overline{c}, p^{\rho_2}\big)
\theta\left(\sigma_1\frac{2(um)_{1/2}}{cp^n}
+\sigma_2\frac{2(u\ell)_{1/2}}{cp^r}\right),
\end{align*}
with $\rho_1$ and $\rho_2$ as in the statement of the lemma.
Here, we replaced for convenience $(umc^{-2})_{1/2}$ and $(u\ell c^{-2})_{1/2}$ 
by $(um)_{1/2}c^{-1}$ and $(u\ell)_{1/2}c^{-1}$, respectively. 
Since they differ only by a unit factor $\delta \in \{\pm1\}$  that is the same for all
$u\in \mathbb{Z}_p^{\times 2}$ and consequently absorbed by the $\varepsilon$-sum. 
Also, note that the $\epsilon$-term only depends on $u \bmod p^{\rho_1}$ and $u \bmod p^{\rho_2}$.

Using~\eqref{chi expansion} and~\eqref{square power series expansion},
it follows that the function $\psi_{m,\ell, c}^{\sigma_1, \sigma_2}(u)$, 
defined for $u\in \mathbb{Z}_p^{\times 2}$ as
\begin{align*}
\psi_{m, \ell, c}^{\sigma_1, \sigma_2}(u) : = \;&\chi (u) 
\epsilon\big(\sigma_1(um)_{1/2}\overline{c}, p^{\rho_1}\big)
\epsilon\big(\sigma_2(u\ell)_{1/2}\overline{c}, p^{\rho_2}\big)
\theta\left(\sigma_1\frac{2(u m)_{1/2}}{cp^n}
+\sigma_2\frac{2(u\ell)_{1/2}}{cp^r}\right),
\end{align*}
satisfies
\begin{align*}
\psi_{m, \ell, c}^{\sigma_1, \sigma_2}(u+p^\kappa t) = \psi_{m,\ell, c}^{\sigma_1, \sigma_2}(u)
\theta\bigg(&\frac{\alpha}{p^n}\Big(\frac{1}{u}p^\kappa t
-\frac{1}{2u^2}p^{2\kappa} t^2\Big)+\frac{\sigma_1}{cp^n}
\Big(\frac{1}{(mu)_{1/2}}p^\kappa mt - \frac{1}{4(mu)_{1/2}^3} p^{2\kappa} m^2 t^2\Big)
\\&+\frac{\sigma_2}{cp^r}\Big(\frac{1}{(u\ell)_{1/2}}p^\kappa \ell t 
- \frac{1}{4(u\ell)_{1/2}^3} p^{2\kappa} \ell^2 t^2\Big)\bigg)\\
= \psi_{m, \ell, c}^{\sigma_1, \sigma_2}(u)
\theta\bigg(&\Big(\frac{\alpha}{u}+\frac{\sigma_1m}{(mu)_{1/2}c}
+\frac{\sigma_2\ell p^{n-r}}{(u\ell)_{1/2}c}\Big)p^{\kappa-n} t\\&
-\Big(\frac{\alpha}{2u^2}+\frac{\sigma_1m^2}{4(mu)_{1/2}^3c}
+\frac{\sigma_2\ell^2p^{n-r}}{4(u\ell)_{1/2}^3c}\Big)p^{2\kappa-n} t^2\bigg),
\end{align*}
for every $t\in \Z_p$ and every $3\kappa \geq n+\iota$.
We now apply Lemma~\ref{statphase-lemma} (2) to $G(m,\ell,c)$ with $\mu=-n$,
which we may do as long as $n\geq 2\kappa$.
Writing $n=2\nu+\rho_1$, we find that
\begin{align*}
G(m,\ell,c)=p^{r/2}&\sum_{\sigma_1, \sigma_2\in\{\pm1\}}
\sum_{u}\psi_{m, \ell, c}^{\sigma_1, \sigma_2}(u)\epsilon\big(-\overline{2}\alpha, p^{\rho_1}\big)\\
&\times\theta\bigg(\frac{u^2}{\alpha}\Big(\frac{\alpha}{u}+\frac{\sigma_1m}{(mu)_{1/2}c}
+\frac{\sigma_2\ell p^{n-r}}{(u\ell)_{1/2}c}\Big)^2p^{-n}\bigg),
\end{align*}
where summation is over all $u\bmod p^\nu$, such that
\begin{equation*}
\frac{\alpha}{u}+\frac{\sigma_1m}{(mu)_{1/2}c}
+\frac{\sigma_2\ell p^{n-r}}{(\ell u)_{1/2}c} 
\in p^{n-\nu-\rho_1}\Z_p=p^{\nu}\Z_p.
\end{equation*}
This condition can be written equivalently as
\begin{equation*}
\alpha c+\sigma_1(mu)_{1/2}
+\sigma_2(\ell u)_{1/2}p^{n-r}
\equiv 0 \bmod p^{\nu}.
\end{equation*}
This congruence has one solutions, yields 
the corresponding stationary point $u$:
\begin{equation*}
u \equiv \left(\frac{\alpha c}{\sigma_1m_{1/2}
+\sigma_2\ell_{1/2}p^{n-r}}\right)^2 \bmod p^{\nu}.
\end{equation*}
Returning to the result of the stationary phase 
evaluation of $G(m,\ell,c)$, we get that
\begin{align*}
G(m,\ell,c) = \;& p^{r/2}\chi(\alpha^2c^2)
\theta\left(-\frac{2\alpha}{p^n}\right)
\epsilon\big(-\overline{2}\alpha, p^{\rho_1}\big)
\sum_{\sigma_1, \sigma_2\in\{\pm1\}}
\epsilon\big(\sigma_1\alpha, p^{\rho_1}\big)\\&\times
\epsilon\big(\sigma_2\alpha(\overline{m}\ell)_{1/2}, p^{\rho_2}\big)
\overline{\chi}\big((\sigma_1m_{1/2}+\sigma_2\ell_{1/2}p^{n-r})^2\big).
\end{align*}
This evaluation of $G(m,\ell,c)$ is equivalent to the statement of the lemma.
\end{proof}

\subsubsection{$p$-adic van der Corput estimates}
The character sum $\mathcal{C}(\tilde{m}, \ell_1, \ell_2)$, introduced in~\eqref{C(m)}, will be estimated 
using Lemma~\ref{Second derivative test}, the Second Derivative Test.
\begin{lemma}\label{p-adic van der Corput estimates}
Let $\mathcal{C}(\tilde{m}, \ell_1, \ell_2)$ be defined as in~\eqref{C(m)}. 
Then $\mathcal{C}(\tilde{m}, \ell_1, \ell_2)$ vanishes unless $\ell_1, \ell_2 \in \Z_p^{\times 2}$ and
$p^{n-r} \mid \tilde{m}$, in which case, we have
$$\mathcal{C}(\tilde{m}, \ell_1, \ell_2) \ll  p^{n+(r+\Omega)/2},$$
where $\Omega={\rm ord}_p(\ell_1-\ell_2)$.
\end{lemma}
\begin{proof}
For $\sigma_1, \sigma_2 \in \{\pm 1\}$ and $u\in \Z_p^{\times2}$, 
we define
\begin{align}\label{G}
\Phi_{\ell_1, \ell_2}^{\sigma_1, \sigma_2}(u)=
\overline{\chi}\big((\sigma_1u_{1/2}+\sigma_2\ell_{1/2}p^{n-r})^2\big)
\chi\big((\sigma_1u_{1/2}+\sigma_2\ell_{1/2}p^{n-r})^2\big),
\end{align}
and 
\begin{align*}
\mathcal{W}(\alpha, c)=\chi(\alpha^2c^2)\overline{\chi}(\alpha^2c^2)
\epsilon\big(-\overline{2}\alpha, p^{\rho_1}\big)\overline{\epsilon}\big(-\overline{2}\alpha, p^{\rho_1}\big).
\end{align*}
By Lemma~\ref{K(m,l,c)},  $\mathcal{C}(\tilde{m}, \ell_1, \ell_2)$ vanishes 
unless $\ell_1, \ell_2 \in \Z_p^{\times 2}$. 
Under the validity of this condition, we derive that
\begin{equation*}
\mathcal{C}(\tilde{m}, \ell_1, \ell_2)=p^r \mathcal{W}(\alpha, c) \sum_{\sigma_1, \sigma_2 \in \{\pm 1\}}
\sum_{u \bmod p^n}\boldsymbol{\epsilon}_{\ell_1, \ell_2}^{\sigma_1, \sigma_2}(u)
\Phi_{\ell_1, \ell_2}^{\sigma_1, \sigma_2}(u)e\left(\frac{\tilde{m}u}{dp^n}\right),
\end{equation*}
where
\begin{equation*}
\boldsymbol{\epsilon}_{\ell_1, \ell_2}^{\sigma_1, \sigma_2}(u)
=\epsilon\big(\sigma_1\alpha, p^{\rho_1}\big)
\epsilon\big(\sigma_2\alpha(\overline{u}\ell_1)_{1/2}, p^{\rho_2}\big)
\overline{\epsilon}\big(\sigma_1\alpha, p^{\rho_1}\big)
\overline{\epsilon}\big(\sigma_2\alpha(\overline{u}\ell_2)_{1/2}, p^{\rho_2}\big)
\end{equation*}
depends only on the class of $u$ modulo $p$; in particular, 
$\boldsymbol{\epsilon}_{\ell_1, \ell_2}^{\sigma_1, \sigma_2}(u+p^\kappa t)=
\boldsymbol{\epsilon}_{\ell_1, \ell_2}^{\sigma_1, \sigma_2}(u)$ for every
$\kappa \geq 1$ and every $t\in \mathbb{Z}_p$.
Let, for brevity, 
\begin{equation*}
\eta=\eta(u)=\sigma_1u_{1/2}+\sigma_2\ell_{1/2}p^{n-r}.
\end{equation*}
Recalling the expansion~\eqref{square power series expansion}, we find that
\begin{equation*}
\eta(u+p^\kappa t)=\eta(u)+\frac{\sigma_1}{2u_{1/2}}p^\kappa t 
- \frac{\sigma_1}{8u_{1/2}^3} p^{2\kappa} t^2 +\mathbf{M}_{p^{3\kappa}}.
\end{equation*}
Fourther, recalling also~\eqref{chi expansion}, we conclude that
\begin{align}\label{eta}
\overline{\chi}(\eta^2(u+p^\kappa t))&=\overline{\chi}(\eta^2) 
\theta \bigg(-\frac{2\alpha}{\eta p^n}\bigg(\frac{\sigma_1}{2 u_{1/2}}p^\kappa t 
- \frac{\sigma_1}{8 u_{1/2}^3} p^{2\kappa} t^2\bigg) +\frac{2\alpha }{2\eta^2 p^n}
\bigg(\frac{\sigma_1}{2 u_{1/2}}p^\kappa t \bigg)^2 + \mathbf{M}_{p^{3\kappa-n}}\bigg)\nonumber
\\&=\overline{\chi}(\eta^2)\theta \bigg(-\frac{\sigma_1\alpha}{\eta p^n u_{1/2}}p^\kappa t
+\frac{\alpha(\sigma_1\eta + u_{1/2}) }{4\eta^2 p^n u_{1/2}^3}p^{2\kappa} t^2
+\mathbf{M}_{p^{3\kappa-n}}\bigg).
\end{align}
Combining~\eqref{G} and~\eqref{eta} , for every $u\in \mathbb{Z}_p^{\times}$ and every $\kappa \geq 1$,
we find that the function 
$\Phi_{\ell_1, \ell_2}^{\sigma_1, \sigma_2}(u)$ satisfies
\begin{align*}
\Phi_{\ell_1, \ell_2}^{\sigma_1, \sigma_2}(u+p^\kappa t)=
\Phi_{\ell_1, \ell_2}^{\sigma_1, \sigma_2}(u)
\theta \bigg(&-\frac{\sigma_1\alpha}{\eta_1 p^n u_{1/2}}p^\kappa t
+\frac{\sigma_1\alpha}{\eta_2 p^n u_{1/2}}p^\kappa t\\
&+\frac{\alpha(\sigma_1\eta_1 + u_{1/2}) }{4\eta_1^2 p^n u_{1/2}^3}p^{2\kappa} t^2
-\frac{\alpha(\sigma_1\eta_2 + u_{1/2}) }{4\eta_2^2 p^n u_{1/2}^3}p^{2\kappa} t^2
+\mathbf{M}_{p^{3\kappa-n}}\bigg),
\end{align*}
where $\eta_1=\eta_1(u)=\sigma_1u_{1/2}+\sigma_2(\ell_1)_{1/2}p^{n-r}$ and
 $\eta_2=\eta_2(u)=\sigma_1u_{1/2}+\sigma_2(\ell_2)_{1/2}p^{n-r}$.
We write above equations as 
\begin{align*}
\Phi_{\ell_1, \ell_2}^{\sigma_1, \sigma_2}(u+p^\kappa t)=
\Phi_{\ell_1, \ell_2}^{\sigma_1, \sigma_2}(u)
\theta \bigg(&\psi_{1, \ell_1, \ell_2}^{\sigma_1, \sigma_2}(u)p^\kappa t
+\psi_{2, \ell_1, \ell_2}^{\sigma_1, \sigma_2}(u)p^{2\kappa} t^2
+\mathbf{M}_{p^{3\kappa-n}}\bigg),
\end{align*}
with the obvious definitions for $\psi_{1, \ell_1, \ell_2}^{\sigma_1, \sigma_2}(u)$
 and $\psi_{2, \ell_1, \ell_2}^{\sigma_1, \sigma_2}(u)$.
We note that, for every $\kappa \geqslant 1$ and every $u \in \Z_p^{\times 2}$,
\begin{equation*}
(u+p^\kappa t)\cdot\left(\frac{1}{u}-\frac{1}{u^{2}}p^\kappa t + \frac{1}{u^3}
p^{2\kappa} t^2\right) \in 1 + p^{3\kappa}\Z_p,
\end{equation*}
so that
\begin{equation*}
(u+p^\kappa t)^{-1} \equiv \frac{1}{u}-\frac{1}{u^{2}}p^\kappa t + \frac{1}{u^3}
p^{2\kappa} t^2  \;\bmod\; p^{3\kappa}.
\end{equation*}
It follows that
\begin{equation*}
\eta^{-1}=(\sigma_1u_{1/2}+\sigma_2\ell_{1/2}p^{n-r})^{-1} 
\equiv \frac{1}{\sigma_1u_{1/2}}-\frac{\sigma_2\ell_{1/2}}{u}p^{n-r} + \frac{\ell}{\sigma_1u_{1/2}^3}
p^{2n-2r}  \;\bmod\; p^{3n-3r}.
\end{equation*}
and
\begin{equation*}
\eta^{-2}=(\sigma_1u_{1/2}+\sigma_2\ell_{1/2}p^{n-r})^{-2} 
\equiv \frac{1}{u}-\frac{2\sigma_2\ell_{1/2}}{\sigma_1u_{1/2}^3}p^{n-r} + \frac{3\ell}{u^2}
p^{2n-2r}  \;\bmod\; p^{3n-3r}.
\end{equation*}
In particular, we show that
\begin{align*}
&\psi_{1, \ell_1, \ell_2}^{\sigma_1, \sigma_2}(u) =
\frac{\sigma_1\sigma_2\alpha}{u_{1/2}^3} p^{-r} 
\big((\ell_1)_{1/2}-(\ell_2)_{1/2}\big)
-\frac{\alpha}{u^2}p^{n-2r}(\ell_1-\ell_2)+\mathbf{M}_{p^{2n-3r}},\\
&\psi_{2, \ell_1, \ell_2}^{\sigma_1, \sigma_2}(u) =-\frac{3\sigma_1\sigma_2\alpha}{4u_{1/2}^5} p^{-r} \big((\ell_1)_{1/2}-(\ell_2)_{1/2}\big)
+\frac{\alpha}{u^3}p^{n-2r}(\ell_1-\ell_2)+\mathbf{M}_{p^{2n-3r}}.
\end{align*}
By applying Lemma~\ref{statphase-lemma} (1), we obtain 
\begin{equation*}
\mathcal{C}(\tilde{m}, \ell_1, \ell_2)=
p^{r+n-\kappa} \mathcal{W}(\alpha, c) \sum_{u} 
\boldsymbol{\epsilon}_{\ell_1, \ell_2}^{\sigma_1, \sigma_2}(u)
\Phi_{\ell_1, \ell_2}^{\sigma_1, \sigma_2}(u)
\theta\left(\frac{\tilde{m}u}{dp^n}\right),
\end{equation*}
where summation is over all $u\bmod p^\kappa$, such that
\begin{equation*}
\frac{\tilde{m}u}{d} +\frac{\sigma_1\sigma_2\alpha}{u_{1/2}^3} 
p^{n-r} \big((\ell_1)_{1/2}-(\ell_2)_{1/2}\big) \in p^{n-\kappa}\Z_p.
\end{equation*}
This condition can be written equivalently as
\begin{equation*}
\tilde{m}u_{1/2}^5+\sigma_1\sigma_2\alpha d
p^{n-r} \big((\ell_1)_{1/2}-(\ell_2)_{1/2}\big) 
\equiv 0  \bmod p^{n-\kappa},
\end{equation*}
which implies $p^{n-r} \mid \tilde{m}$ if we assume $r\geq n/2$.

In order to use Lemma~\ref{Second derivative test}, it remains to 
find ${\rm ord}_p\psi_{1, \ell_1, \ell_2}^{\sigma_1, \sigma_2}(u)$ 
and ${\rm ord}_p\psi_{2, \ell_1, \ell_2}^{\sigma_1, \sigma_2}(u)$. 
Using the fact~\eqref{ord-sqrt}, we have
\begin{equation*}
\text{ord}_p\big((\ell_1)_{1/2}-(\ell_2)_{1/2}\big)=\text{ord}_p(\ell_1-\ell_2).
\end{equation*}
Let $\Omega=\text{ord}_p(\ell_1-\ell_2)$, we get
$${\rm ord}_p\psi_{1, \ell_1, \ell_2}^{\sigma_1, \sigma_2}(u)=
{\rm ord}_p\psi_{2, \ell_1, \ell_2}^{\sigma_1, \sigma_2}(u)=\Omega-r.
$$ 
Applying Lemma~\ref{Second derivative test} with
\begin{align*}
\upsilon=\lambda=r-\Omega, \quad \Omega(\kappa)=\min(\Omega-r+3\kappa, 0)
\quad \omega=\frac{\tilde{m}}{dp^n}, 
\end{align*}
\begin{align*}
\kappa_0=1, \quad \kappa_1=\max(\lambda/2, 1)=\frac{r-\Omega}{2},
\end{align*}
we obtain
$$ \mathcal{C}(\tilde{m}, \ell_1, \ell_2)
\ll p^r(p^n+p^{r-\Omega}) p^{-(r-\Omega)/2}\ll p^{n+(r+\Omega)/2},$$
which immediately implies the announced bound.
\end{proof}

\section{The end game}

Now we are ready to prove Proposition~\ref{main prop}.
By Lemma~\ref{integral:lemma}, Lemma~\ref{zero} and 
Lemma~\ref{p-adic van der Corput estimates}, we get
\begin{align}\label{T estimates}
\mathcal{T}(L,M)\ll
\frac{N^{\varepsilon}M}{dp^n}\bigg( & p^{n+r}
\delta (\ell_1\equiv \ell_2\bmod p^r)
+p^{n+r-1}\delta (\ell_1\equiv \ell_2\bmod p^{r-1})\nonumber\\&
+\sum_{\substack{\Omega \ll \log L \\ p^{\Omega}\|\ell_1-\ell_2}}
\sum_{\substack{0\neq |\tilde{m}|\leq N^{\varepsilon}dp^nC/(cM)\\ p^{n-r}|\tilde{m}}}
p^{n+(r+\Omega)/2}\bigg).
\end{align}
By plugging~\eqref{T estimates} into~\eqref{L estimates} and using the estimate
$\lambda_g(\ell_1)\overline{\lambda_g(\ell_2)}\ll |\lambda_g(\ell_1)|^2+|\lambda_g(\ell_2)|^2$, we have
\begin{align*}
\mathcal{L}(L,M)\ll  \frac{N^{\varepsilon}M}{dp^n} &
\sum_{\ell_1\sim L }\frac{|\lambda_g(\ell_1)|^2}{\ell_1^{1/2}}
\bigg(\sum_{\substack{\ell_2\sim L \\ \ell_2\equiv \ell_1\bmod dp^r}}
 \frac{1}{\ell_2^{1/2}}p^{n+r} + \sum_{\substack{\ell_2\sim L 
 \\ \ell_2\equiv \ell_1\bmod dp^{r-1}}}
 \frac{1}{\ell_2^{1/2}}p^{n+r-1}\bigg)\\
&+\frac{N^{\varepsilon}M}{dp^n}\sum_{\Omega\ll \log L}
\sum_{\ell_1\sim L} \frac{|\lambda_g(\ell_1)|^2}{\ell_1^{1/2}}
  \sum_{\substack{\ell_2\sim L \\ \ell_2\equiv \ell_1\bmod dp^{\Omega}}}
  \frac{1}{\ell_2^{1/2}}\frac{dC}{cM}p^{n+(3r+\Omega)/2}.
\end{align*}
Recall that $M \ll p^{2n-r+\varepsilon}$ and $L\ll p^{r+\varepsilon}$, we get
\begin{align*}
\mathcal{L}(L,M) \ll & \frac{N^{\varepsilon}M}{d}p^r\bigg(1+\frac{L}{dp^r}\bigg)+
  \frac{N^{\varepsilon}M}{d}\bigg(\frac{dC}{cM}+\frac{CL}{cM}\bigg)p^{3r/2}\\
\ll&\frac{N^{\varepsilon}}{d}\bigg(p^{2n}+\frac{C}{c}p^{5r/2}\bigg).
\end{align*}
Therefore, one has
\begin{align*}
\widetilde{S}(N)&\ll \frac{N^{1/2+\varepsilon}}{p^{r/2}}
\sum_{\substack{c\leq C \\ (c,p)=1}}c^{-1}
\sum_{d|c}d^{1/2}\bigg(p^n+\frac{C^{1/2}}{c^{1/2}}p^{5r/4}\bigg)\\
&\ll N^{3/4+\varepsilon}\big(p^{n-3r/4}+p^{r/2}\big).
\end{align*}
As we point out in Remark~\ref{remark1}, the remaining cases are similar and in fact easier. 
Hence, this completes the proof of Proposition~\ref{main prop}.



\begin{thebibliography}{ALMW22}

\bib{ASS}{article} {
    AUTHOR = {Acharya, Ratnadeep },
    AUTHOR = {Sharma, Prahlad},
    AUTHOR = {Singh, Saurabh Kumar},
     TITLE = {{$t$}-aspect subconvexity for {$GL(2)\times GL(2)$}
              {$L$}-function},
   JOURNAL = {J. Number Theory},
  FJOURNAL = {Journal of Number Theory},
    VOLUME = {240},
      YEAR = {2022},
     PAGES = {296--324},
      ISSN = {0022-314X,1096-1658},
   MRCLASS = {11F66 (11F55)},
  MRNUMBER = {4458242},
MRREVIEWER = {Fei\ Hou},
       DOI = {10.1016/j.jnt.2022.01.011},
       URL = {https://doi.org/10.1016/j.jnt.2022.01.011},
}

\bib{BM}{article}{
    AUTHOR = {Blomer, Valentin},
    AUTHOR = {Mili\'cevi\'c, Djordje},
     TITLE = {{$p$}-adic analytic twists and strong subconvexity},
   JOURNAL = {Ann. Sci. \'Ec. Norm. Sup\'er. (4)},
  FJOURNAL = {Annales Scientifiques de l'\'Ecole Normale Sup\'erieure.
              Quatri\`eme S\'erie},
    VOLUME = {48},
      YEAR = {2015},
    NUMBER = {3},
     PAGES = {561--605},
      ISSN = {0012-9593,1873-2151},
   MRCLASS = {11F66 (11L40)},
  MRNUMBER = {3377053},
MRREVIEWER = {Guangshi\ L\"u},
       DOI = {10.24033/asens.2252},
       URL = {https://doi.org/10.24033/asens.2252},
}

\bib{Hux2}{book}{
   author={Huxley, M. N.},
   title={Area, lattice points, and exponential sums},
   series={London Mathematical Society Monographs. New Series},
   volume={13},
   note={Oxford Science Publications},
   publisher={The Clarendon Press, Oxford University Press, New York},
   date={1996},
   pages={xii+494},
   isbn={0-19-853466-3},
}

\bib{HM}{article} {
    AUTHOR = {Harcos, Gergely},
    AUTHOR = {Michel, Philippe},
     TITLE = {The subconvexity problem for {R}ankin-{S}elberg
              {$L$}-functions and equidistribution of {H}eegner points.
              {II}},
   JOURNAL = {Invent. Math.},
  FJOURNAL = {Inventiones Mathematicae},
    VOLUME = {163},
      YEAR = {2006},
    NUMBER = {3},
     PAGES = {581--655},
      ISSN = {0020-9910,1432-1297},
   MRCLASS = {11F66 (11F67 11M41)},
  MRNUMBER = {2207235},
MRREVIEWER = {K.\ Soundararajan},
       DOI = {10.1007/s00222-005-0468-6},
       URL = {https://doi.org/10.1007/s00222-005-0468-6},
}

\bib{HSZ}{article} {
    AUTHOR = {Huang, Bingrong},
    AUTHOR = {Sun, Qingfeng},
    AUTHOR = {Zhang, Huimin},
     TITLE = {Analytic twists of {$\rm GL_2\times GL_2$} automorphic forms},
   JOURNAL = {Math. Nachr.},
  FJOURNAL = {Mathematische Nachrichten},
    VOLUME = {296},
      YEAR = {2023},
    NUMBER = {6},
     PAGES = {2366--2394},
      ISSN = {0025-584X,1522-2616},
   MRCLASS = {11F66 (11F30 11L07)},
  MRNUMBER = {4626837},
MRREVIEWER = {Subhajit\ Jana},
       DOI = {10.1002/mana.202100550},
       URL = {https://doi.org/10.1002/mana.202100550},
}

\bib{HT}{article} {
    AUTHOR = {Holowinsky, Roman},
    AUTHOR = {Templier, Nicolas},
     TITLE = {First moment of {R}ankin-{S}elberg central {$L$}-values and
              subconvexity in the level aspect},
   JOURNAL = {Ramanujan J.},
  FJOURNAL = {Ramanujan Journal. An International Journal Devoted to the
              Areas of Mathematics Influenced by Ramanujan},
    VOLUME = {33},
      YEAR = {2014},
    NUMBER = {1},
     PAGES = {131--155},
      ISSN = {1382-4090,1572-9303},
   MRCLASS = {11F67 (11F11 11L07)},
  MRNUMBER = {3142436},
MRREVIEWER = {Shin-ya\ Koyama},
       DOI = {10.1007/s11139-012-9454-y},
       URL = {https://doi.org/10.1007/s11139-012-9454-y},
}
	
\bib{HX}{article}{
    AUTHOR = {Huang, Bingrong},
    AUTHOR = {Xu, Zhao},
     TITLE = {Hybrid subconvexity bounds for twists of {${\rm GL}(3) \times
              {\rm GL}(2)$} {$L$}-functions},
   JOURNAL = {Algebra Number Theory},
  FJOURNAL = {Algebra \& Number Theory},
    VOLUME = {17},
      YEAR = {2023},
    NUMBER = {10},
     PAGES = {1715--1752},
      ISSN = {1937-0652,1944-7833},
   MRCLASS = {11F66 (11F67)},
  MRNUMBER = {4643300},
MRREVIEWER = {Zhi\ Qi},
       DOI = {10.2140/ant.2023.17.1715},
       URL = {https://doi.org/10.2140/ant.2023.17.1715},
}
		
\bib{IK}{book}{
   author={Iwaniec, Henryk},
   author={Kowalski, Emmanuel},
   title={Analytic number theory},
   series={American Mathematical Society Colloquium Publications},
   volume={53},
   publisher={American Mathematical Society, Providence, RI},
   date={2004},
   pages={xii+615},
   isbn={0-8218-3633-1},
   doi={10.1090/coll/053},
}

\bib{KMS}{article} {
    AUTHOR = {Kumar, Sumit},
    AUTHOR = {Mallesham, Kummari},
    AUTHOR = {Singh, Saurabh Kumar},
     TITLE = {Sub-convexity bound for {$GL(3) \times GL(2)$}
              {$L$}-functions: the depth aspect},
   JOURNAL = {Math. Z.},
  FJOURNAL = {Mathematische Zeitschrift},
    VOLUME = {301},
      YEAR = {2022},
    NUMBER = {3},
     PAGES = {2229--2268},
      ISSN = {0025-5874,1432-1823},
   MRCLASS = {11F66 (11F55 30C99)},
  MRNUMBER = {4437321},
MRREVIEWER = {Prahlad\ Sharma},
       DOI = {10.1007/s00209-022-02993-x},
       URL = {https://doi.org/10.1007/s00209-022-02993-x},
}
	
\bib{KMV}{article}{
   author={Kowalski, E.},
   author={Michel, Ph.},
   author={VanderKam, J.},
   title={Rankin--Selberg $L$-functions in the level aspect},
   journal={Duke Math. J.},
   volume={114},
   date={2002},
   number={1},
   pages={123--191},
   issn={0012-7094},
   doi={10.1215/S0012-7094-02-11416-1},
}

\bib{LMS}{article} {
    AUTHOR = {Lin, Yongxiao},
    AUTHOR = {Michel, Philippe},
    AUTHOR = {Sawin, Will},
     TITLE = {Algebraic twists of {$\rm GL_3\times GL_2$} {$L$}-functions},
   JOURNAL = {Amer. J. Math.},
  FJOURNAL = {American Journal of Mathematics},
    VOLUME = {145},
      YEAR = {2023},
    NUMBER = {2},
     PAGES = {585--645},
      ISSN = {0002-9327},
   MRCLASS = {11F55 (11F66 11L07 11T23 32N10)},
  MRNUMBER = {4570992},
       DOI = {10.1353/ajm.2023.0015},
       URL = {https://doi.org/10.1353/ajm.2023.0015},
}
		
\bib{LS}{article} {
    AUTHOR = {Lin, Yongxiao},
    AUTHOR = {Sun, Qingfeng},
     TITLE = {Analytic twists of {${\rm GL}_3\times{\rm GL}_2$} automorphic
              forms},
   JOURNAL = {Int. Math. Res. Not. IMRN},
  FJOURNAL = {International Mathematics Research Notices. IMRN},
      YEAR = {2021},
    NUMBER = {19},
     PAGES = {15143--15208},
      ISSN = {1073-7928,1687-0247},
   MRCLASS = {11F30 (11F66 11L07)},
  MRNUMBER = {4324739},
MRREVIEWER = {Keshav\ Aggarwal},
       DOI = {10.1093/imrn/rnaa348},
       URL = {https://doi.org/10.1093/imrn/rnaa348},
}
		
\bib{MD}{article}{
    AUTHOR = {Mili\'cevi\'c, Djordje},
     TITLE = {Sub-{W}eyl subconvexity for {D}irichlet {$L$}-functions to prime power moduli},
   JOURNAL = {Compos. Math.},
  FJOURNAL = {Compositio Mathematica},
    VOLUME = {152},
      YEAR = {2016},
    NUMBER = {4},
     PAGES = {825--875},
      ISSN = {0010-437X,1570-5846},
   MRCLASS = {11L07 (11L03 11L26 11L40 11M06 11S80 26E30)},
  MRNUMBER = {3484115},
MRREVIEWER = {Arnaud\ Chadozeau},
       DOI = {10.1112/S0010437X15007381},
       URL = {https://doi.org/10.1112/S0010437X15007381},
}


\bib{Mun3}{article}{
   author={Munshi, Ritabrata},
   title= {The circle method and bounds for {$L$}-functions, {II}:
              {S}ubconvexity for twists of {${\rm GL}(3)$} {$L$}-functions},,
   journal={Amer. J. Math.},
   volume={137},
   date={2015},
   number={3},
   pages={791--812},
   issn= {0002-9327},
   doi= {10.1353/ajm.2015.0018},
}
		
\bib{Mun1}{article}{
   author={Munshi, Ritabrata},
   title={The circle method and bounds for $L$-functions---III: $t$-aspect
   subconvexity for $GL(3)$ $L$-functions},
   journal={J. Amer. Math. Soc.},
   volume={28},
   date={2015},
   number={4},
   pages={913--938},
   issn={0894-0347},
   doi={10.1090/jams/843},
}

\bib{MS2}{article}{
   author={Miller, Stephen D.},
   author={Schmid, Wilfried},
   title={Automorphic distributions, $L$-functions, and Voronoi summation
   for ${\rm GL}(3)$},
   journal={Ann. of Math. (2)},
   volume={164},
   date={2006},
   number={2},
   pages={423--488},
   issn={0003-486X},
   doi={10.4007/annals.2006.164.423},
}

\bib{MS}{article}{
    AUTHOR = {Munshi, Ritabrata},
    AUTHOR = {Singh, Saurabh Kumar},
     TITLE = {Weyl bound for {$p$}-power twist of {$\rm GL(2)$}
              {$L$}-functions},
   JOURNAL = {Algebra Number Theory},
  FJOURNAL = {Algebra \& Number Theory},
    VOLUME = {13},
      YEAR = {2019},
    NUMBER = {6},
     PAGES = {1395--1413},
      ISSN = {1937-0652,1944-7833},
   MRCLASS = {11F66 (11F55 11M41)},
  MRNUMBER = {3994569},
MRREVIEWER = {Yoshinori\ Yamasaki},
       DOI = {10.2140/ant.2019.13.1395},
       URL = {https://doi.org/10.2140/ant.2019.13.1395},
}

\bib{S}{article}{
    AUTHOR = {Sun, Qingfeng},
     TITLE = {Bounds for {${\rm GL}_2 \times{\rm GL}_2$} {$L$}-functions in
              the depth aspect},
   JOURNAL = {Manuscripta Math.},
  FJOURNAL = {Manuscripta Mathematica},
    VOLUME = {174},
      YEAR = {2024},
    NUMBER = {1-2},
     PAGES = {429--451},
      ISSN = {0025-2611,1432-1785},
   MRCLASS = {11F66},
  MRNUMBER = {4730439},
       DOI = {10.1007/s00229-023-01515-1},
       URL = {https://doi.org/10.1007/s00229-023-01515-1},
}
	
\bib{SZ}{article}{
   author={Sun, Qingfeng},
   author={Zhao, Rui},
   title={Bounds for ${\rm GL}_3$ $L$-functions in depth aspect},
   journal={Forum Math.},
   volume={31},
   date={2019},
   number={2},
   pages={303--318},
}

\bib{WZ}{article}{
    AUTHOR = {Wang, Xin},
    AUTHOR = {Zhu, Tengyou},
     TITLE = {Hybrid subconvexity bounds for twists of {${\rm GL}(3)$}
              {$L$}-functions},
   JOURNAL = {Int. J. Number Theory},
  FJOURNAL = {International Journal of Number Theory},
    VOLUME = {20},
      YEAR = {2024},
    NUMBER = {2},
     PAGES = {413--439},
      ISSN = {1793-0421,1793-7310},
   MRCLASS = {11F30 (11F11 11L07)},
  MRNUMBER = {4709634},
       DOI = {10.1142/S1793042124500210},
       URL = {https://doi.org/10.1142/S1793042124500210},
}
		
\end{thebibliography}
\end{document}